\documentclass[reqno]{amsart}
\usepackage{amsaddr}
\usepackage{amsfonts, amsmath, amssymb, amsthm}

\usepackage{mathptmx}

\usepackage[top=3cm, bottom=3cm, left=3cm, right=3cm]{geometry}

\usepackage{multirow,multicol,diagbox,slashbox}
\usepackage{algorithm}
\usepackage[table]{xcolor}
\newcommand{\tabincell}[2]{\begin{tabular}{@{}#1@{}}#2\end{tabular}}
\usepackage[square,sort,comma,numbers]{natbib}
\newcommand{\cg}{\cellcolor{gray}} 

\newtheorem{theorem}{Theorem}
\newtheorem{theorem2}{Theorem}
\newtheorem{cor}{Corollary}
\newtheorem{lemma}{Lemma}

\newtheorem{defn}{Definition}

\newtheorem{prob}{Problem}

\theoremstyle{definition}
\newtheorem{example}{Example}
\newtheorem*{remark}{Remark}
\theoremstyle{remark}

\usepackage{enumitem}

\usepackage[T1]{fontenc}
\DeclareMathAlphabet{\mathpzc}{OT1}{pzc}{m}{it}
\usepackage{yfonts}
\usepackage[sans]{dsfont}
\usepackage{txfonts}
\usepackage[mathscr]{euscript}
\usepackage{bbm}

\newcommand{\grp}[1]{\mathpzc{#1}}	
\newcommand{\set}[1]{\mathpzc{#1}}	

\newcommand{\itr}{\mathbb Z}			

\newcommand{\ff}[1]{{\mathbb F}_{#1}}		
\newcommand{\ffs}[1]{{\mathbb F}_{#1}^\star}	
\newcommand{\ffx}[1]{\ff{#1}[x]}		

\DeclareMathOperator{\image}{Im}
\newcommand{\im}[1]{\image(#1)}

\DeclareMathOperator{\preim}{PreIm}

\DeclareMathOperator{\res}{Res}			
\DeclareMathOperator{\pres}{P-Res}		

\DeclareMathOperator{\range}{range}

\makeatletter
\@namedef{subjclassname@2020}{%
  \textup{2020} Mathematics Subject Classification}
\makeatother

\subjclass[2020]{
11T06, 
11T71, 
12E10, 
12E20, 
90C10  
}
\date{}

\begin{document}

\title{Algorithms for computing the permutation resemblance of functions over finite groups}
\author{Li-An Chen and Robert S. Coulter}
\address{Department of Mathematical Sciences\\
University of Delaware\\
Newark DE 19716, USA}
\email[L.-A. Chen]{lianchen@udel.edu}
\email[R.S. Coulter]{coulter@udel.edu}

\begin{abstract}
Permutation resemblance measures the distance of a function from being a
permutation. Here we show how to determine the permutation resemblance through
linear integer programming techniques.
We also present an algorithm for constructing feasible solutions to this
integer program, and use it to prove an upper bound for permutation
resemblance for some special functions.
Additionally, we present a generalization of the linear integer program that
takes a function on a finite group and determines a permutation with
the lowest differential uniformity among those most resembling it.
\end{abstract}

\maketitle

\section{Introduction and the Main Results}

Throughout this paper $\grp{G}$ denotes a finite group, not necessarily
abelian but written additively, and $\ff{q}$ denotes the finite field of order
$q$.
For a finite set $\set{S}$, $\# \set{S}$ denotes the cardinality of $\set{S}$.

Let $f:\grp{G}\to \grp{G}$.
The set of distinct images of $f$ is denoted by
$\im{f}:=\{f(x)\,:\,x\in\grp{G}\}$, and we write $V(f):=\# \im{f}$.
We call $f$ a \textit{permutation} over $\grp{G}$ if $\im{f}=\grp{G}$, i.e.
when $f$ is a bijection over $\grp{G}$.
For $b\in \grp{G}$, the set of preimages of $b$ under $f$ is denoted by
$\preim(f,b):=\{x\in \grp{G}\,:\, f(x)=b\}$. The \textit{uniformity of $f$} is
defined by \[u(f):=\max_{b\in \grp{G}} \#\preim(f,b).\]
For a nonzero $a\in \grp{G}$, the \textit{differential operator of $f$ in the direction of $a$} is defined by 
\begin{equation*}
\Delta_{f,a}(x):=f(x+a)-f(x), 
\end{equation*}
and the \textit{differential uniformity (DU) of $f$} is defined by 
\begin{equation*}
\delta_{f}:=\max_{a\in \grp{G}\setminus\{0\}}u(\Delta_{f,a}).  
\end{equation*}
The concept of DU was first suggested by Nyberg \cite{nyberg93}.
The lower the DU, the more resistant $f$ is to differential attacks when used in a cryptosystem.
Functions with optimal DU are called \textit{planar} when $\# \grp{G}$ is odd (with $1$-DU). A classic example is $x^2$, which is planar over every field of odd characteristic. When $\# \grp{G}$ is even, the optimal functions are called \textit{almost perfect nonlinear (APN)} (with $2$-DU). One of the most important problems related to DU is the construction of permutations over finite fields with optimal or low DU. These functions are the most desired for the construction of S-boxes in cryptosystems.
Planar functions cannot be permutations, so over
$\ff{q}$ when $q$ is odd the problem becomes that of finding permutations with
low DU.
For extensions of $\ff{2}$, while there are several examples of APN permutations over odd dimensions, there is only one known APN permutation in even
dimensions, an example in dimension $6$ discovered by Browning, Dillon,
McQuistan, and Wolfe in 2010 \cite{bro10}.
For more than a decade, the existence of APN permutations over higher even extensions of $\ff{2}$ remains open and is considered one of the most important open problems in the theory of APN functions. 

Motivated by the construction of low DU permutations, the authors introduced
the notions of permutation resemblance in \cite{prestheory}.
The concept provides a new way to measure the ``distance'' of a function from
being a permutation. For two functions $f,h:\grp{G}\to \grp{G}$, we first define the \textit{resemblance of $f$ to $h$} by
\begin{equation*}
\res(f,h)=V(f-h). 
\end{equation*}
Observe that for any functions $f,h:\grp{G}\to \grp{G}$, we have $\res(f,f)=\# \{0\}=1$, $\res(f,h)=\res(h,f)$, and $\res(f,h+c)=\res(f,h)$ for any constant $c\in \grp{G}$. The central concept of this paper is the following. 
\begin{defn}
For $f:\grp{G}\to \grp{G}$, the \textit{permutation resemblance of $f$} is defined by 
\begin{equation*}
\pres(f)=\min\{\res(f,h)\,:\, h\text{ is a permutation over }\grp{G}\}, 
\end{equation*}
or equivalently, by writing $f-h=-g$, 
\begin{equation}\label{eq:def_ppres}
\pres(f)=\min\{V(g)\,:\, g+f\text{ is a permutation over }\grp{G}\}. 
\end{equation}
\end{defn}

It is important to note that there could be many functions ($h$ or $g$ in the
above expressions) that give the $\pres(f)$.
Given a non-permutation $f$, $\pres(f)$ can be understood as the ``minimum'' changes required in order to modify $f$ into a permutation.
Unlike other methods for measuring the distance a function is from being a
permutation ($V(f)$ is a common one), permutation resemblance can be used to
construct low DU permutations if we start with function $f$ with low or optimal
DU.
In \cite{prestheory}, it was proved that if $f,g$ are two functions over a finite abelian group $\grp{G}$, then 
\[\delta_{g+f}\le \delta_{f}\cdot \big(V(g)^{2}-V(g)+1\big). \]
In particular, if $V(g)=\pres(f)$, then 
\[\delta_{g+f}\le \delta_{f}\cdot \big(\pres(f)^{2}-\pres(f)+1\big). \]
Thus, by computing $\pres(f)$ for an optimal DU function $f$, and finding a set of such $g$, we obtain a set of permutations of the form $g+f$ whose DU is controlled by $\delta_{f}$ and $\pres(f)$. This observation motivates the present
article.

The aim of this article is to present algorithms for computing $\pres(f)$ for
an arbitrary $f$ and for constructing permutations $g+f$ that satisfy
$V(g)=\pres(f)$.
The key idea is based on rephrasing the problem of determining $\pres(f)$ into
the problem of searching a certain family of subtables in the {\em subtraction
table}
indexed by $\grp{G}$ and $\im{f}$, see Section \ref{sc:sub_tb} for details.
In Section \ref{sc:alg_IP}, we further formulate this new problem as a linear integer programming problem (IP).
Though it is not often used in the research of permutations over finite groups,
linear programming is known to be a useful tool in extremal combinatorics,
a recent example being the paper by Wagner \cite{wagner20} which disproved a
number of open conjectures in extremal combinatorics with linear programming
methods. 

By solving the IP, one obtains the exact value of $\pres(f)$ for any $f:\grp{G}\to \grp{G}$, and a set of permutations $g+f$ such that $V(g)=\pres(f)$.
In Section \ref{sc:alg_avg}, we present an algorithm for constructing a feasible solution of this IP. Using this algorithm, we then prove the following upper bounds for $\pres(f)$ when $f$ is a two-to-one function. 

\begin{theorem}\label{th:2to1bound_even}
If $\# \grp{G}=q$ is even and $f: \grp{G}\to \grp{G}$ is two-to-one, then
\[\pres(f)\le \left\lceil 2\sqrt{q} \right\rceil-1. \]
When $q$ is a perfect square, the bound can be improved to 
\[\pres(f)\le  2\sqrt{q}-2.\]
\end{theorem}

\begin{theorem}\label{th:2to1bound_odd}
If $\# \grp{G}=q$ is odd and $f:\grp{G}\to \grp{G}$ such that $f(0)=0$ and $f$ is two-to-one on the nonzero elements. Then 
\[\pres(f)\le \left\lceil 2\sqrt{q-1}\right\rceil-1.\]
When $q-1$ is a perfect square, the bound can be improved to 
\[\pres(f)\le  2\sqrt{q-1}-2.\]
\end{theorem}

In Section \ref{sc:alg_ex}, we discuss some conditions when a $k$-subset of $\grp{G}$ could be a candidate of the image of $g:\grp{G}\to \grp{G}$ such that $g+f$ is a permutation, and give an algorithm to test this. Finally, we close the article by generalizing the IP in Section \ref{sc:alg_IP} into a formulation that can optimize DU. This IP has a large number of variables and constraints, but it can combine both requirements of $\pres(f)$ and DU, and we believe it has significant potential to be used/adapted to create examples under other optimal
measurements.  We also give some computational results concerning $\pres(x^d)$
over $\ff{q}$ with $\gcd(d,q-1)>1$ in an appendix.

\section{The subtraction table}\label{sc:sub_tb}

Let $\# \grp{G}=q$ and $f:\grp{G}\to \grp{G}$.
The \textit{subtraction table of $f$} is a table $M_{f}$ with $q$ rows and $V(f)$ columns. The rows of $M_{f}$ are indexed by the $q$ elements of $\grp{G}$, and the columns are indexed by the $V(f)$ elements of $\im{f}$. For $r\in \grp{G}$ and $c\in \im{f}$, the entry $m_{r,c}$ of $M_{f}$ at row $r$, column $c$, is defined by $m_{r,c}:=r-c$. 

For $t\in\itr$, $0\le t\le u(f)$, denote the set of elements of $\grp{G}$ with exactly $t$ preimages under $f$ by \[P_{t}:=\{b\in \grp{G} \,:\, \# \preim(f,b)=t\}.\] In particular, $P_{0}$ is the set of all non-images of $f$, and $\im{f}=\bigcup_{t=1}^{u(f)}P_{t}$. We order the columns of $M_f$ by $P_{1},P_{2},\dots,P_{u(f)}$, and the rows by $P_{1},P_{2},\dots,P_{u(f)}, P_{0}$, where elements within the same set $P_{t}$ can be listed in arbitrary order but in the same way for both rows and columns. With this ordering, the diagonal of the upper part of $M_f$ indexed by $P_{1},P_{2},\dots,P_{u(f)}$ is all $0$. 
This table is useful when working with the sum $g+f$ for some $g:\grp{G}\to \grp{G}$, since if $g(x)+f(x)=b$ for some $x,b\in  \grp{G}$, then the value of $g(x)$ is the value of the corresponding entry $m_{b,f(x)}$ of $M_{f}$. 

For the remaining, we call a collection of entries of $M_{f}$ a \textit{subtable of $M_{f}$}. We do not assume that a subtable is a block, unless otherwise specified. Next, we define a family of subtables of $M_{f}$ important for our algorithms. Let $S$ be a $k$-subset of $\grp{G}$, and $\mathcal{A}$ be a collection of entries of $M_{f}$. We call $\mathcal{A}$ an \textit{admissible subtable of $M_{f}$ with value set $S$} if it satisfies the following conditions:

\begin{itemize}
\item[(A1)] $S$ is the set of values of the entries in $\mathcal{A}$.
\item[(A2)] For every $c\in \im{f}$, there are exactly $ \# \preim(f,c)$ distinct $r$ such that $m_{r,c}\in\mathcal{A}$, i.e., every column $c$ of $M_f$ has exactly $\#\preim(f,c)$ distinct entries in $\mathcal{A}$. 
\end{itemize}

We may simply use the term ``admissible subtable'' when there is no danger of confusion. Condition (A2) implies that $\mathcal{A}$ has exactly $\sum_{c\in\im{f}} \# \preim(f,c)=q$ entries. Note that for (A1) and (A2) to be true, we must have $k\ge u(f)$ since every element of $\grp{G}$ appears exactly once in each column of $M_f$. 

Admissible subtables are crucial for our algorithms since each of them corresponds to a set of functions with known image sets. This correspondence is presented in the following lemma and its proof. 

\begin{lemma}\label{lm:subtb_correspondence}
Let $f:\grp{G}\to \grp{G}$ and let $M_{f}=(m_{r,c})$ be its subtraction table. Let $S$ be a $k$-subset of $\grp{G}$ such that $k\ge u(f)$. Then there is a one to $\prod_{t=1}^{u(f)}t!^{\# P_{t}}$ correspondence between the set of all admissible subtables of $M_{f}$ with value set $S$, and the set of all functions $g:\grp{G}\to \grp{G}$ such that
\begin{enumerate}[label=(\alph*)]
\item $\im{g}=S$; 
\item $g$ is injective on $\preim(f,c)$ for every $c\in\im{f}$, i.e., $g(x)\neq g(y)$ whenever $x\neq y$ and $f(x)=f(y)$. 
\end{enumerate}

Moreover, if two functions $g$ and $h$ correspond to the same admissible subtable, then $\im{g}=\im{h}$, and $\im{g+f}=\im{h+f}$. 
\end{lemma}

\begin{proof}
Let $g:\grp{G}\to \grp{G}$ be a function that satisfies both conditions (a) and (b). We associate to $g$ a subtable $\mathcal{A}$ by choosing $m_{r,c}\in\mathcal{A}$ if and only if there is an $x\in \grp{G}$ such that $f(x)=c$ and $g(x)+f(x)=r$. By the definition of $M_{f}$, this implies that $g(x)=r-c=m_{r,c}$. Since every column is a complete set of $\grp{G}$, we may always find such $m_{r,c}$ for every $x\in \grp{G}$. So condition (A1) is satisfied. Now suppose $x_1\neq x_2$ and $f(x_1)=f(x_2)=c$ for some $c\in\im{f}$. Let $r_1=g(x_1)+c$ and $r_2=g(x_2)+c$. Then since $g$ is injective on the set $\preim(f,c)$, we have $r_1\neq r_2$. Therefore, for all $c\in\im{f}$, every $x\in \preim(f,c)$ determines a unique entry of $\mathcal{A}$, which means (A2) is satisfied. 

Conversely, let $\mathcal{A}$ be an admissible subtable with value set $S$. For each $c\in\im{f}$, since by (A2) there are exactly $\# \preim(f,c)$ entries of column $c$ in $\mathcal{A}$, we can define a function $g:\grp{G}\to \grp{G}$ that satisfies conditions (a) and (b) by setting $g(x)=m_{r,c}\in\mathcal{A}$ for a unique $r$ for each $x\in \preim(f,c)$. Clearly, there are $\# \preim(f,c)!$ ways to define $g$ like this on the set $\preim(f,c)$. Hence, the number of such $g$ is
\begin{equation*}
\begin{aligned}
\prod_{c\in\im{f}}\# \preim(f,c)!=\prod_{t=1}^{u(f)}\prod_{c\in P_{t}}t!=\prod_{t=1}^{u(f)}t!^{\# P_{t}}.
\end{aligned}
\end{equation*}

Finally, by our construction, $\im{g+f}=\{m_{r,c}+f(x)\,:\, m_{r,c}\in\mathcal{A}\}=\{r\in \grp{G}\,:\, m_{r,c}\in\mathcal{A}\}$. Therefore, the image sets of $g$ and $g+f$ only depend on $\mathcal{A}$. 
\end{proof}

For an admissible subtable $\mathcal{A}$, we define
\[\range(\mathcal{A}):=\{r\in \grp{G}\,:\, m_{r,c}\in\mathcal{A}\}.\]
That is, $\range(\mathcal{A})$ is the set of indices of the rows of $M_{f}$
which have at least one entry in $\mathcal{A}$. As shown in the proof of Lemma \ref{lm:subtb_correspondence}, $\range(\mathcal{A})$ is equal to $\im{g+f}$ for any corresponding $g$. In particular, if $\range(\mathcal{A})=\grp{G}$, then any corresponding $g$ satisfies $\im{g+f}=\grp{G}$ and therefore $g+f$ is a permutation. This gives the following useful corollary. 
\begin{cor}\label{co:ub_ppres_subtb_correspondence}
If $f:\grp{G}\to \grp{G}$, then $\pres(f)\le k$ if and only if there exists an admissible subtable of $M_{f}$ whose range is $\grp{G}$ and value set is a $k$-subset of $\grp{G}$. 
\end{cor}
\begin{proof}
By the definition of $\pres(f)$ in the form of (\ref{eq:def_ppres}), $\pres(f)\le k$ if and only if there exists a function $g:\grp{G}\to \grp{G}$ such that $V(g)=k$ and $g+f$ is a permutation. This implies that $g$ is injective on $\preim(f,c)$ for all $c\in \grp{G}$, since otherwise we would have $(g+f)(x)=(g+f)(y)$ for some $x,y\in\preim(f,c)$. By Lemma \ref{lm:subtb_correspondence}, $g$ corresponds to an admissible subtable $\mathcal{A}$ with value set $\im{g}$ and $\range(\mathcal{A})=\im{g+f}=\grp{G}$. 
\end{proof}

By Corollary \ref{co:ub_ppres_subtb_correspondence}, the problem of determining $\pres(f)$ can be rephrased as follows. 
\begin{prob}\label{pb:pres_table}
Given $f:\grp{G}\to \grp{G}$, find an admissible subtable $\mathcal{A}$ of $M_f$ such that $\range(\mathcal{A})=\grp{G}$ and the value set $S$ has the minimum cardinality.
\end{prob}

We close this section by considering $f(x)=x^2$ over $\ff{9}$, giving
$M_{f}$, an example of admissible subtable, and the correspondence described in Lemma \ref{lm:subtb_correspondence}. 

\begin{example}\label{ex:mf_tables}
Let $\alpha$ be a primitive element of $\ff{9}$ which is a root of $x^2+2x+2$. Let $f(x)=x^2\in\ffx{9}$. Then 
$\im{f}=\{0,\pm 1, \pm\alpha^2\}$ (note that $\alpha^{4}=-1$), and $V(f)=5$. The preimage sets are the following: $\preim(f,0)=\{0\}$, $\preim(f,1)=\{\pm 1\}$, $\preim(f,\alpha^2)=\{\pm\alpha\}$, $\preim(f,-1)=\{\pm \alpha^2\}$, and $\preim(f,-\alpha^2)=\{\pm\alpha^3\}$. Hence, $P_{0}=\{\pm\alpha,\pm\alpha^3\}$, $P_{1}=\{0\}$, and $P_{2}=\{\pm 1, \pm\alpha^2\}$. 

The subtraction table $M_{f}$ is shown in Table \ref{tb:sub_x2_adm}. The shaded entries in Table \ref{tb:sub_x2_adm} provide an admissible subtable $\mathcal{A}$ of $M_{f}$ with value set $S=\{0,1,\alpha\}$. 
\begin{table}[h]
\caption{\label{tb:sub_x2_adm} The subtraction table $M_{f}$ of $f=x^2$ over $\ff{9}$ and an admissible subtable. The first two columns and rows label the coordinates of $M_{f}$. }
\begin{center}
\begin{tabular}{|c|c|c c c c c|}
\hline
\multicolumn{2}{|c|}{\multirow{2}{*}{\backslashbox{$r\searrow$}{$c\searrow$}}} & \multicolumn{1}{|c|}{$P_{1}$} &\multicolumn{4}{c|}{$P_{2}$}\\\cline{3-7}
 \multicolumn{2}{|c|}{}&\multicolumn{1}{|c|}{$0$} & $1$ & $\alpha^2$ & $-1$ & $-\alpha^2$ \\\hline
$P_{1}$& $0$ & \cg$0$ & $-1$ & $-\alpha^2$ & $1$ & $\alpha^2$ \\\cline{1-2}
\multirow{4}{*}{$P_{2}$} & $1$&$1$ & $0$ & $-\alpha$ & $-1$ & $-\alpha^3$\\ 
 & $\alpha^2$ & $\alpha^2$ & \cg$\alpha$ & \cg$0$ & $-\alpha^3$ & $-\alpha^2$ \\
 & $-1$ & $-1$ & \cg$1$& $\alpha^3$ & \cg$0$ & $\alpha$\\
 & $-\alpha^2$ & $-\alpha^2$ & $\alpha^3$ & $\alpha^2$ & $-\alpha$ & \cg$0$\\\cline{1-2}
\multirow{4}{*}{$P_{0}$} & $\alpha$& $\alpha$ & $-\alpha^3$ & $-1$ & $\alpha^2$ & $\alpha^3$\\ 
 & $\alpha^3$ & $\alpha^3$ & $-\alpha$ & $\alpha$ & $-\alpha^2$ & $-1$ \\
 & $-\alpha$ & $-\alpha$ & $-\alpha^2$ & $-\alpha^3$ & $\alpha^3$ & \cg$1$\\
 & $-\alpha^3$ & $-\alpha^3$ & $\alpha^2$ & \cg$1$ & \cg$\alpha$ & $-\alpha$\\\hline
\end{tabular}
\end{center}
\end{table}

A function $g$ that corresponds to $\mathcal{A}$ can be defined by $g(0)=0$; $g(1)=\alpha$, $g(-1)=1$, $g(\alpha)=0$, $g(-\alpha)=1$, $g(\alpha^2)=0$, $g(-\alpha^2)=\alpha$, $g(\alpha^3)=0$, and $g(-\alpha^3)=1$. Switching any values within the same preimage set, for example, take $g(1)=1$, $g(-1)=\alpha$ instead, does not change $\im{g}=S$ and $\im{g+f}=\{0,\alpha^2,-1,-\alpha^2,-\alpha,-\alpha^3\}=\range(\mathcal{A})$. 
\end{example}

\section{Linear Integer Programming Approach}\label{sc:alg_IP}

As stated in Problem \ref{pb:pres_table}, determining $\pres(f)$ of a function $f:\grp{G}\to\grp{G}$ is equivalent to finding an admissible subtable $\mathcal{A}$ of $M_f$ whose value set has the minimum cardinality. This problem can be phrased as the following binary linear integer program: 


\begin{align}
&\text{minimize:}\qquad   \sum_{v\in \grp{G}}y_{v} \notag\\
&\text{subject to:} \notag\\
& \sum_{c\in\im{f}}x_{r,c}=1, \text{ for all }r\in \grp{G},\label{eq:IP_row}\\
& \sum_{r\in \grp{G}}x_{r,c}=\# \preim(f,c), \text{ for all }c\in\im{f},\label{eq:IP_col}\\
& x_{r,c}\le y_{v}, \text{ for all }r\in \grp{G},\,c\in\im{f},\,v\in \grp{G}\text{ such that }r-c=v,\label{eq:IP_countxy}\\
& x_{r,c},y_{v}\in\{0,1\}, \text{ for all }r\in \grp{G},\,c\in\im{f},\,v\in \grp{G}. \notag
\end{align}

We associate every entry $m_{r,c}$ with a $\{0,1\}$-valued variable $x_{r,c}$. These variables record the coordinates of the entries of $\mathcal{A}$. If $m_{r,c}\in \mathcal{A}$, then $x_{r,c}=1$; otherwise, $x_{r,c}=0$. Note that by (A2), an admissible subtable must have exactly $q$ entries. So $\range(\mathcal{A})=\grp{G}$ if and only if each row has exactly one entry in $\mathcal{A}$. Thus, we add constraints (\ref{eq:IP_row}) to make sure that $\range(\mathcal{A})=\grp{G}$. Similarly, for each column $c\in\im{f}$, we require (\ref{eq:IP_col}) so that each column $c$ has exactly $\# \preim(f,c)$ entries in $\mathcal{A}$. 
To count $\# S$, we associate a $\{0,1\}$-valued variable $y_{v}$ for every $v\in \grp{G}$. If $v\in S$, then $y_{v}=1$; otherwise, $y_{v}=0$. Hence, the objective is to minimize $\sum_{v}y_{v}$. Finally, we add the constraints (\ref{eq:IP_countxy}), which force $y_{v}=1$ whenever any entry $m_{r,c}$ with value $v$ is chosen in $\mathcal{A}$. 

This IP has a total of $q(V(f)+1)$ binary variables, $q+V(f)$ equality constraints, and $qV(f)$ inequality constraints. For functions over a finite field of order $q$, we use Magma \cite{magma} to generate $M_f$, and use Gurobi via Python interface \cite{gurobi} to solve the IP. Some computational results using
the algorithm are given in an appendix.

\section{Algorithm for Constructing an Admissible Subtable}\label{sc:alg_avg}

We now present an algorithm for constructing an admissible subtable $\mathcal{A}$ such that $\range(\mathcal{A})=\grp{G}$.
This gives a feasible solution of the IP in Section \ref{sc:alg_IP} and thus an upper bound for $\pres(f)$. We first describe the algorithm for a general $f$, and then focus on the special case when $f$ is two-to-one, a specific class
of functions with strong connections to functions with optimal DU. 

\subsection{General case}\label{ssb:alg_a_general}
Let $\# \grp{G}=q$, $f:\grp{G}\to \grp{G}$ and $u(f)=u$. The main idea of the algorithm is to iteratively choose a value $v$ that appears at least the average number of times in the subtraction table, and then append those entries with value $v$ to $\mathcal{A}$. 

The algorithm is described below, followed by an explanation.
We use $()$ to describe a table or subtable, and $\{\}$ to describe a set. For example, $M_f=(m_{r,c} \,:\, r\in \grp{G},c\in\im{f})$ is the subtraction table that has $qV(f)$ entries, and $S=\{m_{r,c} \,:\, r\in \grp{G},c\in\im{f}\}$ is the set of all distinct values of the entries $m_{r,c}$ that has $q$ elements. 

\begin{algorithm}[H]\caption{Constructing an admissible subtable for a given function $f$. }\label{alg:average}
\begin{itemize}
\item[\textbf{Input:}] $M_{f}=(m_{r,c}\,:\, r\in \grp{G},c\in\im{f})$, some positive integer $k\le q$. 
\item[\textbf{Initialize:}] $S_{0}=\{0\}$, $\mathcal{A}_{0}=(m_{c,c}\,:\, c\in\im{f})$, $M_{0}$ as shown in Table \ref{tb:alg_a_m0} of size $n_0\times n_0$, where $n_{0}=\# P_0=q-V(f)$. 

\item[\textbf{Step i.}] For $1\le i\le k$, repeat the following steps as long as $n_{i-1}>0$: 
\begin{enumerate}
\item Let $\mu_{i-1}$ be the average number of appearance in $M_{i-1}$ over all elements of $\grp{G}$, round up to the nearest integer. Pick one element $v_i\in\grp{G}$ such that $v_{i}$ has at least $\mu_{i-1}$ appearance in $M_{i-1}$. 
\item Let $B_{i}$ be a subtable of $M_{i-1}$ that collects $\mu_{i-1}$ entries with value $v_{i}$. If more than one column has $v_{i}$ in the same row, only the one from the smallest second index may be selected. 
\item Let $M_{i}$ be the table obtained from $M_{i-1}$ by deleting every row and column that has an entry in $B_{i}$. 
\item Let $S_{i}=S_{i-1}\cup\{v_{i}\}$, $\mathcal{A}_{i}=\mathcal{A}_{i-1}\cup B_i$ and $n_{i}=n_{i-1}-\mu_{i-1}$. 
\end{enumerate}
\item[\textbf{Return:}] $S=S_{k}\cup\{m_{c,c}\in M_{k}\}$, $\mathcal{A}=\mathcal{A}_k\cup (m_{c,c}\in M_{k})$, where $(m_{c,c}\in M_{k})$ is the diagonal of $M_k$. 
\end{itemize}
\end{algorithm}
To simplify the arguments, we say a row or column of $M_f$ is in $P_j$ for some $0\le j\le u$, if it is indexed by an element in $P_j$. To construct an admissible subtable $\mathcal{A}$, we need to choose $j$ entries from each column of $M_f$ that is in $P_{j}$ for some $1\le j\le u$ to satisfy condition (A2). To make sure that $\range(\mathcal{A})=\grp{G}$, each row of $M_f$ has to be chosen exactly once. We first choose the entry with value $0$ from each column. By our definition of $M_f$, this is the diagonal of the upper part of $M_f$ whose rows and columns are indexed by elements of $\im{f}$. Thus, our initial subtable is $\mathcal{A}_0=(m_{c,c}\,:\, c\in\im{f})$ with value set $S_{0}=\{0\}$. 

Now, each row in $P_1,\dots,P_u$ in $M_f$ has been chosen once, so the remaining entries should be chosen from the rows in $P_0$. Also, each column has been chosen once, so $j-1$ more entries should be chosen from every column in $P_j$ for all $2\le j\le u$. This is equivalent to making a selection from each row and each column exactly once from the table $M_0$ that is shown in Table \ref{tb:alg_a_m0}. The table $M_0$ is constructed by gluing $j-1$ copies of the block of the subtable of $M_{f}$ whose rows are in $P_{0}$ and columns are in $P_{j}$ for all $2\le j\le u$. We assign each block a second index for the columns. For example, $P_{j,1},\dots,P_{j,j-2},P_{j,j-1}$ are the $j-1$ copies of $P_j$. 

\begin{table}[h]\caption{The initial table $M_{0}$ for the general case}\label{tb:alg_a_m0}
\begin{center}
\begin{tabular}{c|p{2cm}|p{1.5cm}|p{1.5cm}|c|p{1cm}|c|p{1cm}|p{.8cm}|c|p{.8cm}|}
\multicolumn{1}{c}{}&\multicolumn{1}{c}{$P_{2,1}$}&\multicolumn{1}{c}{$P_{3,1}$}&\multicolumn{1}{c}{$P_{3,2}$}&\multicolumn{1}{c}{$\cdots$}&\multicolumn{1}{c}{$P_{u-1,1}$}&\multicolumn{1}{c}{$\cdots$}&\multicolumn{1}{c}{$P_{u-1,u-2}$}&\multicolumn{1}{c}{$P_{u,1}$}&\multicolumn{1}{c}{$\cdots$}&\multicolumn{1}{c}{$P_{u,u-1}$} \\
\cline{2-11}
$P_{0}$ &\tabincell{r}{\vspace{.3cm}\\$\phantom{v_1}$\\\vspace{.3cm}}&  & &$\cdots$ &  & $\cdots$&  & & $\cdots$& \\\cline{2-11}
\end{tabular}
\end{center}

\end{table}

By definition of the $P_{j}$'s, we have $q=\sum_{j=0}^{u}\# P_{j}=\sum_{j=1}^{u}j\# P_{j}$. Therefore, 
\begin{equation*}
\begin{aligned}
\# P_{0}=\sum_{j=1}^{u}(j-1)\# P_{j}=\sum_{j=2}^{u}(j-1)\# P_{j}.
\end{aligned}
\end{equation*}
Hence, $M_{0}$ is an $n_{0}\times n_{0}$ square table where $n_{0}=\# P_{0}=q-V(f)$. 
To complete the construction of $\mathcal{A}$, we repeat the following steps. First, define the number of appearance of an element $v\in \grp{G}$ in a table $M$ by \[\#\{r\in \grp{G} \,:\, \exists c\in\im{f} \text{ such that } v=m_{r,c}\in M\}.\] Note that it only counts once if $v$ appears in more than one column of the same row. In fact, in that case, all such columns are originally identical copies of the same column when we construct $M_0$. 

Let $\mu_0$ be the average number (round up to the nearest integer) of appearance in $M_{0}$ over all elements $ \grp{G}$. Pick an element $v_{1}\in \grp{G}$ that has at least $\mu_0$ appearance in $M_{0}$ and let the new value set be $S_{1}=S_{0}\cup\{v_{1}\}$. Define a subtable $B_1$ of $M_0$ by choosing exactly $\mu_0$ entries with value $v_1$ from $M_0$. If $u>2$ and $v_1$ appears in more than one column of the same row, only the entry of the column in the smallest second index of $P_{j,i}$'s may be chosen. The new subtable $\mathcal{A}_1$ is then defined by appending $B_1$ to $\mathcal{A}_0$. The new square table $M_1$ is obtained from $M_0$ by deleting the rows and columns that has an entry in $B_1$. For example, suppose $v_1$ is in some row $r$ of blocks $P_{3,1}$ and $P_{3,2}$, as shown in Table \ref{tb:generalAG}. These blocks were initially identical, and therefore the column indices of $v_1$ in these blocks are the same element $c\in P_3$. We delete row $r$ and only the column $c$ in $P_{3,1}$, and append entry $m_{r,c}$ to $\mathcal{A}_0$. 

\begin{table}[h]\caption{An example of the general case. }\label{tb:generalAG}
\begin{center}
\begin{tabular}{c|p{2cm}|p{1.5cm}|p{1.5cm}|c|p{1cm}|c|p{1cm}|p{.8cm}|c|p{.8cm}|}
\multicolumn{1}{c}{}&\multicolumn{1}{c}{$P_{2,1}$}&\multicolumn{1}{c}{$P_{3,1}$}&\multicolumn{1}{c}{$P_{3,2}$}&\multicolumn{1}{c}{$\cdots$}&\multicolumn{1}{c}{$P_{u-1,1}$}&\multicolumn{1}{c}{$\cdots$}&\multicolumn{1}{c}{$P_{u-1,u-2}$}&\multicolumn{1}{c}{$P_{u,1}$}&\multicolumn{1}{c}{$\cdots$}&\multicolumn{1}{c}{$P_{u,u-1}$} \\
\multicolumn{1}{c}{}&\multicolumn{1}{c}{}&\multicolumn{1}{c}{\tabincell{r}{\hspace{.7cm}$c$}}&\multicolumn{1}{c}{\tabincell{r}{\hspace{.7cm}$c$}}&\multicolumn{1}{c}{}&\multicolumn{1}{c}{}&\multicolumn{1}{c}{}&\multicolumn{1}{c}{}&\multicolumn{1}{c}{}&\multicolumn{1}{c}{}&\multicolumn{1}{c}{} \\\cline{2-11}
$P_{0}\quad r$  & & \tabincell{r}{\vspace{.3cm}\\\hspace{1cm}$v_1$\\\vspace{.3cm}} & \tabincell{r}{\vspace{.3cm}\\\hspace{1cm}$v_1$\\\vspace{.3cm}} &$\cdots$ &  & $\cdots$&  & & $\cdots$& \\\cline{2-11}
\end{tabular}
\end{center}
\end{table}

Since exactly $\mu_0$ rows and $\mu_0$ columns are deleted from $M_0$, $M_1$ is an $n_1\times n_1$ square table where $n_1=n_0-\mu_0$. This completes the first step. We repeat this step using $M_1$, $\mathcal{A}_1$, $S_1$ instead, and obtain an $n_2\times n_2$ square table $M_2$, a subtable $\mathcal{A}_2$ with value set $S_2$, and so on. This process can continue as long as the square table $M_{i-1}$ obtained from the previous iteration is nonempty, i.e., $n_{i-1}>0$. Finally, assume that we repeat this step $k$ times and obtain an $n_k\times n_k$ square table $M_k$ and a subtable $\mathcal{A}_k$ with value set $S_k$. An admissible subtable $\mathcal{A}$ can then be completed by appending the diagonal of $M_k$ to $\mathcal{A}_k$, where its value set $S$ is the union of $S_k$ and the values of these diagonal entries. 

We now give an example of the algorithm in action when $f$ is two-to-one (a
class of functions we are especially interested in), see Table \ref{tb:alg1_ex}.
\begin{table}\caption{An example of applying Algorithm \ref{alg:average} for a two-to-one function $f$ over a group $\grp{G}=\{0,a_{1},a_{2},\dots,a_{7}\}$, where $\im{f}=P_{2}=\{0,a_{1},a_{2},a_{3}\}$, $P_{0}=\grp{G}\setminus\im{f}=\{a_{4},a_{5},a_{6},a_{7}\}$. }\label{tb:alg1_ex}
\begin{center}
\begin{tabular}{|m{8cm}|c|}
\hline
\multicolumn{1}{|c|}{The $k$-th step}
&
The table $M_{k-1}$
\\\hline
\textbf{Initialize:} Here is $M_f$. Let $S_{0}=\{0\}$ and $\mathcal{A}_0$ be the shaded entries. $M_{0}$ is the lower half of the table. The side length of $M_{0}$ is $n_{0}=4$. 
&
\begin{tabular}{|c|c|c c c c|}
\hline
\multicolumn{2}{|c|}{\multirow{2}{*}{\backslashbox{$r\searrow$}{$c\searrow$}}}  & \multicolumn{4}{c|}{$P_{2}$}\\\cline{3-6}
 \multicolumn{2}{|c|}{}& $0$ & $a_{1}$ & $a_{2}$ & $a_{3}$ \\\hline
\multirow{4}{*}{$P_{2}$} & $0$ & \cg$0$ & $*$ & $*$ & $*$\\ 
 & $a_{1}$ & $*$ & \cg$0$ & $*$ & $*$ \\
 & $a_{2}$ & $*$ & $*$ & \cg$0$ & $*$\\
 & $a_{3}$ & $*$ & $*$ & $*$ & \cg$0$\\\cline{1-6}
\multirow{4}{*}{\tabincell{c}{$P_{0}$}} & $a_{4}$ & $a_{4}$ & $a_{2}$ & $a_{1}$ & $a_{6}$\\  
 & $a_{5}$ & $a_{5}$ & $a_{7}$ & $a_{3}$ & $a_{2}$ \\
 & $a_{6}$ & $a_{6}$ & $a_{5}$ & $a_{7}$ & $a_{4}$\\
 & $a_{7}$ & $a_{7}$ & $a_{3}$ & $a_{6}$ & $a_{1}$\\\hline
\end{tabular}
\\\hline
\textbf{Step 1:} $\mu_0=\lceil 16/8\rceil=2$, so we can take $S_{1}=\{0,a_{4}\}$, and obtain $\mathcal{A}_{1}$ by adding the shaded entries to $\mathcal{A}_{0}$. $M_{1}$ is the table formed by deleting row $a_{4},a_{6}$ and column $0,a_{3}$, and the side length of $M_{1}$ is $n_{1}=n_0-\mu_0=2$.
&
\begin{tabular}{|c |c|c c c c|}
\hline
\multicolumn{2}{|c|}{\multirow{2}{*}{\backslashbox{$r\searrow$}{$c\searrow$}}}  & \multicolumn{4}{c|}{$P_{2}$}\\\cline{3-6}
 \multicolumn{2}{|c|}{}& $0$ & $a_{1}$ & $a_{2}$ & $a_{3}$ \\\hline
\multirow{4}{*}{\tabincell{c}{$P_{0}$}} & $a_{4}$ & \cg$a_{4}$ & $a_{2}$ & $a_{1}$ & $a_{6}$\\  
 & $a_{5}$ & $a_{5}$ & $a_{7}$ & $a_{3}$ & $a_{2}$ \\
 & $a_{6}$ & $a_{6}$ & $a_{5}$ & $a_{7}$ & \cg$a_{4}$\\
 & $a_{7}$ & $a_{7}$ & $a_{3}$ & $a_{6}$ & $a_{1}$\\\hline
\end{tabular}\\\hline
\textbf{Step 2:} $\mu_1=\lceil 4/8\rceil=1$, so we can take $S_{2}=\{0,a_{4},a_{3}\}$, and obtain $\mathcal{A}_{2}$ by adding the shaded entries to $\mathcal{A}_{1}$. $M_{2}$ is the table formed by deleting row $a_{5}$ and column $a_{2}$, and the side length of $M_{2}$ is $n_{2}=n_1-\mu_1=1$.
&
\begin{tabular}{|c |c| c c|}
\hline
\multicolumn{2}{|c|}{\multirow{2}{*}{\backslashbox{$r\searrow$}{$c\searrow$}}}  & \multicolumn{2}{c|}{$P_{2}$}\\\cline{3-4}
 \multicolumn{2}{|c|}{}& $a_{1}$ & $a_{2}$ \\\hline
\multirow{2}{*}{\tabincell{c}{$P_{0}$}} 
 & $a_{5}$ & $a_{7}$ & \cg$a_{3}$ \\
 & $a_{7}$ & $a_{3}$ & $a_{6}$ \\\hline
\end{tabular}\\\hline
\textbf{Stop ($k=2$):} Now we stop and return $S=\{0,a_{4},a_{3}\}$, $\mathcal{A}=\mathcal{A}_{2}\cup \{\text{the shaded entry}\}$. 
&
\begin{tabular}{|c |c|c|}
\hline
\multicolumn{2}{|c|}{\multirow{2}{*}{\backslashbox{$r\searrow$}{$c\searrow$}}}  & \multicolumn{1}{c|}{$P_{2}$}\\\cline{3-3}
 \multicolumn{2}{|c|}{}& $a_{1}$ \\\hline
\multirow{1}{*}{$P_{0}$} 
 & $a_{7}$ & \cg$a_{3}$  \\\hline
\end{tabular}\\\hline
\textbf{Output:} The shaded entries present the output $\mathcal{A}$. This gives an upper bound $\pres(f)\le \# S=3$. Note that the worst case scenario in (\ref{eq:pres_bound_nk}) gives $1+2+n_{2}=4$.
&
\begin{tabular}{|c|c|c c c c|}
\hline
\multicolumn{2}{|c|}{\multirow{2}{*}{\backslashbox{$r\searrow$}{$c\searrow$}}}  & \multicolumn{4}{c|}{$P_{2}$}\\\cline{3-6}
 \multicolumn{2}{|c|}{}& $0$ & $a_{1}$ & $a_{2}$ & $a_{3}$ \\\hline
\multirow{4}{*}{$P_{2}$} & $0$ & \cg$0$ & $*$ & $*$ & $*$\\ 
 & $a_{1}$ & $*$ & \cg$0$ & $*$ & $*$ \\
 & $a_{2}$ & $*$ & $*$ & \cg$0$ & $*$\\
 & $a_{3}$ & $*$ & $*$ & $*$ & \cg$0$\\\cline{1-6}
\multirow{4}{*}{\tabincell{c}{$P_{0}$}} & $a_{4}$ & \cg$a_{4}$ & $a_{2}$ & $a_{1}$ & $a_{6}$\\  
 & $a_{5}$ & $a_{5}$ & $a_{7}$ & \cg$a_{3}$ & $a_{2}$ \\
 & $a_{6}$ & $a_{6}$ & $a_{5}$ & $a_{7}$ & \cg$a_{4}$\\
 & $a_{7}$ & $a_{7}$ & \cg$a_{3}$ & $a_{6}$ & $a_{1}$\\\hline
\end{tabular}\\\hline

\textbf{Alternative stop ($k=1$):} We may also stop right after Step 1 instead, and return $S=\{0,a_{4},a_{6},a_{7}\}$, $\mathcal{A}=\mathcal{A}_{1}\cup\{\text{the shaded entries}\}$. 
&
\begin{tabular}{|c |c| c c|}
\hline
\multicolumn{2}{|c|}{\multirow{2}{*}{\backslashbox{$r\searrow$}{$c\searrow$}}}  & \multicolumn{2}{c|}{$P_{2}$}\\\cline{3-4}
 \multicolumn{2}{|c|}{}& $a_{1}$ & $a_{2}$ \\\hline
\multirow{2}{*}{\tabincell{c}{$P_{0}$}} 
 & $a_{5}$ & \cg$a_{7}$ & $a_{3}$ \\
 & $a_{7}$ & $a_{3}$ & \cg$a_{6}$ \\\hline
\end{tabular}\\\hline
\textbf{Alternative output:} The shaded entries present the alternative output $\mathcal{A}$ if we stop right after Step 1. This gives an upper bound $\pres(f)\le \# S=4$. In this case, $1+1+n_{1}=4$ so the worst case scenario in (\ref{eq:pres_bound_nk}) is met. 
&
\begin{tabular}{|c|c|c c c c|}
\hline
\multicolumn{2}{|c|}{\multirow{2}{*}{\backslashbox{$r\searrow$}{$c\searrow$}}}  & \multicolumn{4}{c|}{$P_{2}$}\\\cline{3-6}
 \multicolumn{2}{|c|}{}& $0$ & $a_{1}$ & $a_{2}$ & $a_{3}$ \\\hline
\multirow{4}{*}{$P_{2}$} & $0$ & \cg$0$ & $*$ & $*$ & $*$\\ 
 & $a_{1}$ & $*$ & \cg$0$ & $*$ & $*$ \\
 & $a_{2}$ & $*$ & $*$ & \cg$0$ & $*$\\
 & $a_{3}$ & $*$ & $*$ & $*$ & \cg$0$\\\cline{1-6}
\multirow{4}{*}{\tabincell{c}{$P_{0}$}} & $a_{4}$ & \cg$a_{4}$ & $a_{2}$ & $a_{1}$ & $a_{6}$\\  
 & $a_{5}$ & $a_{5}$ & \cg$a_{7}$ & $a_{3}$ & $a_{2}$ \\
 & $a_{6}$ & $a_{6}$ & $a_{5}$ & $a_{7}$ & \cg$a_{4}$\\
 & $a_{7}$ & $a_{7}$ & $a_{3}$ & \cg$a_{6}$ & $a_{1}$\\\hline
\end{tabular}\\\hline
\end{tabular}
\end{center}
\end{table}

We remark that the algorithm may be improved by using a different method to choose $v_t$ in each step $t$. For example, we could use a greedy strategy to select a value $v_t$ with the maximum appearance and append all relevant entries to $\mathcal{A}_t$, instead of taking the average appearance only. See Table \ref{tb:alg1_ex_greedy} for an example. We use the average number here in order to estimate $\# S$ in the next section for the special case of two-to-one functions. 
\begin{table}\caption{An example of applying Algorithm \ref{alg:average} for $f$ in Table \ref{tb:alg1_ex} but replacing each step by a greedy strategy. The initialization step is the same as in Table \ref{tb:alg1_ex}. }\label{tb:alg1_ex_greedy}
\begin{center}
\begin{tabular}{|m{8cm}|c|}
\hline
\multicolumn{1}{|c|}{The $k$-th step}
&
The table $M_{k-1}$
\\\hline
\textbf{Step 1': }$a_{6}$ appeared three times which is one of the most, so we can take $S_{1}=\{0,a_{6}\}$, and obtain $\mathcal{A}_{1}$ by adding the shaded entries to $\mathcal{A}_{0}$. $M_{1}$ is the table formed by deleting row $a_{4},a_{6},a_{7}$ and column $0,a_{2},a_{3}$, and the side length of $M_{1}$ is $n_{1}=1$.
&
\begin{tabular}{|c |c|c c c c|}
\hline
\multicolumn{2}{|c|}{\multirow{2}{*}{\backslashbox{$r\searrow$}{$c\searrow$}}}  & \multicolumn{4}{c|}{$P_{2}$}\\\cline{3-6}
 \multicolumn{2}{|c|}{}& $0$ & $a_{1}$ & $a_{2}$ & $a_{3}$ \\\hline
\multirow{4}{*}{\tabincell{c}{$P_{0}$}} & $a_{4}$ & $a_{4}$ & $a_{2}$ & $a_{1}$ & \cg$a_{6}$\\  
 & $a_{5}$ & $a_{5}$ & $a_{7}$ & $a_{3}$ & $a_{2}$ \\
 & $a_{6}$ & \cg$a_{6}$ & $a_{5}$ & $a_{7}$ & $a_{4}$\\
 & $a_{7}$ & $a_{7}$ & $a_{3}$ & \cg$a_{6}$ & $a_{1}$\\\hline
\end{tabular}\\\hline

\textbf{Stop ($k=1$): }Now we stop and return $S=\{0,a_{6},a_{7}\}$, $\mathcal{A}=\mathcal{A}_{1}\cup\{\text{the shaded entries}\}$. 
&
\begin{tabular}{|c |c|c|}
\hline
\multicolumn{2}{|c|}{\multirow{2}{*}{\backslashbox{$r\searrow$}{$c\searrow$}}}  & \multicolumn{1}{c|}{$P_{2}$}\\\cline{3-3}
 \multicolumn{2}{|c|}{}& $a_{1}$ \\\hline
\multirow{1}{*}{$P_{0}$} 
 & $a_{5}$ & \cg$a_{7}$  \\\hline
\end{tabular}\\\hline
\textbf{Output: }The shaded entries present the output $\mathcal{A}$. This gives an upper bound $\pres(f)\le \# S=3$. In this case, $1+1+n_{1}=3$ so the worst case scenario in (\ref{eq:pres_bound_nk}) is met. 
&
\begin{tabular}{|c|c|c c c c|}
\hline
\multicolumn{2}{|c|}{\multirow{2}{*}{\backslashbox{$r\searrow$}{$c\searrow$}}}  & \multicolumn{4}{c|}{$P_{2}$}\\\cline{3-6}
 \multicolumn{2}{|c|}{}& $0$ & $a_{1}$ & $a_{2}$ & $a_{3}$ \\\hline
\multirow{4}{*}{$P_{2}$} & $0$ & \cg$0$ & $*$ & $*$ & $*$\\ 
 & $a_{1}$ & $*$ & \cg$0$ & $*$ & $*$ \\
 & $a_{2}$ & $*$ & $*$ & \cg$0$ & $*$\\
 & $a_{3}$ & $*$ & $*$ & $*$ & \cg$0$\\\cline{1-6}
\multirow{4}{*}{\tabincell{c}{$P_{0}$}} & $a_{4}$ & $a_{4}$ & $a_{2}$ & $a_{1}$ & \cg$a_{6}$\\  
 & $a_{5}$ & $a_{5}$ & \cg$a_{7}$ & $a_{3}$ & $a_{2}$ \\
 & $a_{6}$ & \cg$a_{6}$ & $a_{5}$ & $a_{7}$ & $a_{4}$\\
 & $a_{7}$ & $a_{7}$ & $a_{3}$ & \cg$a_{6}$ & $a_{1}$\\\hline
\end{tabular}\\\hline

\end{tabular}
\end{center}
\end{table}

\section{The special case when $f$ is two-to-one}

We now focus on the special case when $f$ is a two-to-one function.
Our motivation for doing so is directly related to the problem of finding
permutations with low DU in that a substantial class of optimal DU functions,
that is planar functions, are two-to-one.

First, consider the case when $\# \grp{G}=q$ is even and $V(f)=q/2$. The uniformity is $u(f)=2$ and $\im{f}=P_2$. So the initial square table $M_0$ is simply the lower half of $M_f$, i.e., the block whose rows are in $P_0$ and columns are in $P_2$. Since there is no need to repeat any blocks, we omit the second indices of the columns. Using Algorithm \ref{alg:average}, we prove the following upper bound for $\pres(f)$. 
\begin{theorem2}
If $\# \grp{G}=q$ is even and $f: \grp{G}\to \grp{G}$ is two-to-one, then
\[\pres(f)\le \left\lceil 2\sqrt{q} \right\rceil-1. \]
When $q$ is a perfect square, the bound can be improved to 
\[\pres(f)\le  2\sqrt{q}-2.\]
\end{theorem2}
\begin{proof}
We use the notation stated in Algorithm \ref{alg:average}.
Observe that if the steps of Algorithm \ref{alg:average} are iterated $k$ times for some integer $k\ge 1$, then the output value set is $S=\{0\}\cup\{v_1,\dots,v_k\}\cup\{m_{c,c}\in M_k\}$. This implies that for any integer $k\ge 1$, 
\begin{equation}\label{eq:pres_bound_nk}
\pres(f)\le\# S\le 1+k+n_k. 
\end{equation}
In particular, 
\begin{equation}\label{eq:pres_bound_nk_min}
\pres(f)\le \min_{k\ge 1,k\in\itr} \{1+k+n_{k}\}. 
\end{equation}
For the remaining, let $\# \grp{G}=q$ be even and $f: \grp{G}\to \grp{G}$ be two-to-one. Our strategy is to prove a simple upper bound $t_k$ for $n_k/q$ that depends on $k$, and find the minimum value of $1+k+qt_k$ for $k\ge 1$. 


First, since $\#P_0=q-V(f)=q/2$, we have $n_0=q/2$. For $k>0$, the table $M_{k-1}$ has $n_{k-1}^2$ entries, so the average number of appearances in $M_{k-1}$ over all elements of $ \grp{G}$ is $\mu_{k-1}=\lceil n_{k-1}^2/q\rceil$. The table $M_k$ is of dimension $n_k\times n_k$, where $n_k=n_{k-1}-\mu_{k-1}$ since $M_k$ is obtained from deleting exactly $\mu_{k-1}$ rows and $\mu_{k-1}$ columns from $M_{k-1}$. Therefore, we obtain the following recursive definition of $n_k$: 
\begin{equation}\label{eq:recurrence_n}
\begin{aligned}
n_{0}&=\frac{q}{2},\\
n_{k}&=n_{k-1}-\left\lceil\frac{n_{k-1}^2}{q}\right\rceil\text{ for }k\ge 1. 
\end{aligned}
\end{equation}
The sequence $\{n_{k}\}_{k\ge 0}$ is strictly decreasing when $n_k>0$ since we must obtain a smaller table $M_k$ after each step, until reaching $0$. 
Let $s_k=n_k/q$ for $k\ge 0$. Dividing both sides of (\ref{eq:recurrence_n}) by $q$ gives the following recursive definition of $s_k$: 
\begin{equation}\label{eq:recurrence_s}
\begin{aligned}
s_{0}&=\frac{1}{2},\\
s_{k}&=s_{k-1}-\frac{\left\lceil q s_{k-1}^2\right\rceil}{q}\text{ for }k\ge 1. 
\end{aligned}
\end{equation}
The sequence $\{s_{k}\}_{k\ge 0}$ is also strictly decreasing as long as $s_k>0$ since it is just a rescale of $\{n_{k}\}_{k\ge 0}$. Define another sequence $\{t_{k}\}_{k\ge 0}$ by 
\begin{equation*}
\begin{aligned}
t_{k}&=\frac{1}{k+2},\text{ for all }k\ge 0.
\end{aligned}
\end{equation*}
Clearly, the sequence $\{t_{k}\}_{k\ge 0}$ is also strictly decreasing. When $k=0$, $s_{0}=t_{0}=\frac{1}{2}$. For $k\ge 1$, we prove $s_{k}< t_{k}$ by induction on $k$. When $k=1$, we have 
\[s_{1}=\frac{1}{2}-\frac{\left\lceil q/4\right\rceil}{q}\le \frac{1}{2}-\frac{1}{4}=\frac{1}{4}<t_{1}=\frac{1}{3}. \] 
For $k>1$, observe that 
\begin{equation*}
\begin{aligned}
t_{k}-t_{k-1}+t_{k-1}^2&=\frac{1}{k+2}-\frac{1}{k+1}+\frac{1}{(k+1)^2}\\
&=\frac{k^2+2k+1-(k^2+3k+2)+(k+2)}{(k+2)(k+1)^2}\\
&=\frac{1}{(k+2)(k+1)^2}.
\end{aligned}
\end{equation*}
Therefore, 
\begin{equation*}
\begin{aligned}
t_{k}=t_{k-1}-t_{k-1}^2+\frac{1}{(k+2)(k+1)^2}.
\end{aligned}
\end{equation*}
Now assume that the assertion is true for $k-1$. Then
\begin{equation*}
\begin{aligned}
t_{k}-s_{k}&=\left(t_{k-1}-t_{k-1}^2+\frac{1}{(k+2)(k+1)^2}\right)-\left(s_{k-1}-\frac{\left\lceil q s_{k-1}^2\right\rceil}{q}\right)\\
&=t_{k-1}-s_{k-1}-t_{k-1}^2+\frac{\left\lceil q s_{k-1}^2\right\rceil}{q}+\frac{1}{(k+2)(k+1)^2}\\
&\ge t_{k-1}-s_{k-1}-t_{k-1}^2+s_{k-1}^2+\frac{1}{(k+2)(k+1)^2}\\
&=\Big(t_{k-1}-s_{k-1}\Big)\bigg(1-\Big(t_{k-1}+s_{k-1}\Big)\bigg)+\frac{1}{(k+2)(k+1)^2}.
\end{aligned}
\end{equation*}
We have $\frac{1}{(k+2)(k+1)^2}>0$ for $k\ge 1$, and $t_{k-1}-s_{k-1}>0$ by induction hypothesis. Moreover, since both $\{s_{k}\}_{k\ge 0}$ and $\{t_{k}\}_{k\ge 0}$ are decreasing, we have $t_{k-1}\le t_{0}=\frac{1}{2}$ and $s_{k-1}\le s_{0}=\frac{1}{2}$. Hence, $1-(t_{k-1}+s_{k-1})\ge 0$ and therefore $t_{k}-s_{k}>0$ is proved. 
We now apply $n_k=qs_{k}<qt_{k}$ to (\ref{eq:pres_bound_nk_min}) to obtain
\begin{equation}\label{eq:pres_bound_tk}
\begin{aligned}
\pres(f)&\le \min_{k\ge 1,k\in\itr}\{1+k+n_{k}\}< \min_{k\ge 1,k\in\itr}\bigg\{1+k+\frac{q}{k+2}\bigg\}.
\end{aligned}
\end{equation}
By extending the function $h(k)=1+k+\frac{q}{k+2}$ to all real numbers $k\ge 1$ and computing the derivative $\frac{dh}{dk}$, one can show that $h$ achieves its minimum over $k\ge 1$ when $q=(k+2)^2$. If $q$ is a perfect square, then substituting $k=\sqrt{q}-2\in\itr$ into (\ref{eq:pres_bound_tk}) gives the bound
\begin{equation}\label{eq:pres_bound_perfectsq}
\begin{aligned}
\pres(f)&< 1+k+(k+2)=2\sqrt{q}-1. 
\end{aligned}
\end{equation}
If $q$ is not a perfect square, then $k=\sqrt{q}-2\notin\itr$. Let $k^{+}=\lceil k\rceil=\lceil \sqrt{q}\rceil-2$ and $k^{-}=\lfloor k \rfloor=\lfloor \sqrt{q}\rfloor-2$. The bound above still holds as long as the integer part of $h(k)$ and $\min (h(k^{+}),h(k^{-}))$ are the same. In fact, 
\begin{equation*}
\begin{aligned}
|h(k^{-})-h(k)|&=\bigg|\lfloor \sqrt{q} \rfloor+\frac{q}{\lfloor \sqrt{q} \rfloor}-\sqrt{q}-\frac{q}{\sqrt{q}}\bigg|\\
&=\Bigg|q\frac{\sqrt{q}-\lfloor \sqrt{q} \rfloor}{\sqrt{q}\lfloor \sqrt{q} \rfloor}-(\sqrt{q}-\lfloor \sqrt{q} \rfloor)\Bigg|\\
&=\Bigg|(\sqrt{q}-\lfloor \sqrt{q} \rfloor)\bigg(\frac{q}{\sqrt{q}\lfloor \sqrt{q} \rfloor}-1\bigg)\Bigg|<1,
\end{aligned}
\end{equation*}
and similarly one can show that $|h(k^{+})-h(k)|<1$, so (\ref{eq:pres_bound_perfectsq}) holds for most cases. In a rare occasion when the integer part of $h(k)$ is $1$ less than $\min \{h(k^{+}),h(k^{-})\}$, we have the slightly weaker expression 
\begin{equation*}
\begin{aligned}
\pres(f)&\le \lceil h(k)\rceil=\lceil 2\sqrt{q} \rceil-1. 
\end{aligned}
\end{equation*}
\end{proof}
A slight modification of the previous proof yields the following.
\begin{theorem2}
If $\# \grp{G}=q$ is odd and $f:\grp{G}\to \grp{G}$ such that $f(0)=0$ and $f$ is two-to-one on the nonzero elements. Then 
\[\pres(f)\le \left\lceil 2\sqrt{q-1}\right\rceil-1.\]
When $q-1$ is a perfect square, the bound can be improved to 
\[\pres(f)\le  2\sqrt{q-1}-2.\]
\end{theorem2}
\begin{proof}
Let $\# \grp{G}=q$ be odd,   $f(0)=0$ and $f$ be two-to-one on the nonzero elements of $\grp{G}$. In this case, we have $V(f)=(q+1)/2$ so $n_0=q-V(f)=(q-1)/2$. Since $0$ does not appear in $M_0$, we may let $\mu_k$ be the average number (round up to the nearest integer) of appearance of $M_k$ over $\grp{G}\setminus\{0\}$ instead. Then $\mu_{k-1}=\lceil n_{k-1}^2/(q-1)\rceil$, and $n_k=n_{k-1}-\lceil n_{k-1}^2/(q-1)\rceil$ for $k\ge 1$. Similar to the proof of Theorem \ref{th:2to1bound_even}, we consider $s_k=n_{k}/(q-1)$. The sequence $\{s_k\}_{k\ge 0}$ satisfies the same recursive definition in (\ref{eq:recurrence_s}), except that $q$ is replaced by $q-1$. It can be shown that $s_k< t_k=\frac{1}{k+2}$ still holds for $k\ge 1$ by almost the same computation as in the proof of Theorem \ref{th:2to1bound_even}. Therefore, we have
\begin{equation*}
\begin{aligned}
\pres(f)&\le \min_{k\ge 1,k\in\itr}\{1+k+n_{k}\}<\min_{k\ge 1,k\in\itr}\bigg\{1+k+\frac{q-1}{k+2}\bigg\}. 
\end{aligned}
\end{equation*}
Since $h(k)=1+k+\frac{q-1}{k+2}$ achieves its minimum over all real numbers $k\ge 1$ when $q-1=(k+2)^2$, if $q-1$ is a perfect square, we have
\begin{equation*}
\begin{aligned}
\pres(f)&< 1+k+(k+2)=2\sqrt{q-1}-1. 
\end{aligned}
\end{equation*}
Otherwise if $q-1$ is not a perfect square, then by similar arguments as in the proof of Theorem \ref{th:2to1bound_even}, we obtain
\begin{equation*}
\begin{aligned}
\pres(f)&\le \left\lceil h(\sqrt{q-1}-2)\right\rceil=\left\lceil 2\sqrt{q-1}\right\rceil-1.
\end{aligned}
\end{equation*}
This completes the proof.
\end{proof}
An immediate application of this theorem is the following. 
\begin{cor}
If $q$ is odd, and $f=x^2\in\ffx{q}$, then $\pres(f)\le \left\lceil 2\sqrt{q-1}\right\rceil-1$. 
\end{cor}

\section{The Existence of Admissible Subtables for Given Value Sets}\label{sc:alg_ex}

We have seen that an upper bound $\pres(f)\le k$ can be proved by finding an admissible subtable of $M_{f}$ whose value set is a $k$-subset of $ \grp{G}$ and the range is $ \grp{G}$. In this section, we discuss the existence of such admissible subtables when the value set is given. 

Let $f :\grp{G}\to \grp{G}$ and $S$ be a $k$-subset of $\grp{G}$. Clearly, any admissible subtable of $M_{f}$ with value set contained in $S$ must be a subtable of $M_{S}:=(m_{r,c}\in S\,:\, m_{r,c}\in M_f)$, the subtable of all entries with values in $S$. Define the range of $M_{S}$ similarly as
\begin{equation}\label{eq:def_rangeMS}
\range(M_{S}):=\{r\in \grp{G}\,:\, m_{r,c}\in S\},
\end{equation}
i.e., the indices of the rows of $M_f$ that has an entry in $S$. 
Hence, if we want to show $\pres(f)\le k$ by finding an admissible subtable with value set $S$, we should only proceed when $\range(M_{S})=\grp{G}$. Indeed, if $\range(M_{S})\neq \grp{G}$ for every $k$-subset $S$ of $\grp{G}$, then there is no admissible subtable with value set a $k$-subset and range $\grp{G}$. So by Corollary \ref{co:ub_ppres_subtb_correspondence}, $\pres(f)>k$. 

If we view each row of $M_f$ as a set of entries, then $M_f$ can be seen as a hypergraph $H_{f}$ with vertex set $\grp{G}$, and edge set the set of the rows of $M_f$, which are all of cardinality $V(f)$. For a hypergraph $H$, a subset of the vertex set of $H$ is called a \textit{vertex cover} (also known as a \textit{transversal}) if it intersects every edge of $H$. From this point of view, $S$ is a vertex cover of $H_{f}$ if and only if every row of $M_f$ has at least one entry in $S$, i.e., $\range(M_{S})=\grp{G}$. Therefore, it is possible that tools from the theory of hypergraphs may be used to study $\pres(f)$. 

We call a subset $S\subseteq \grp{G}$ a \textit{cover of $\grp{G}$ associated with $f$} if $\range(M_{S})=\grp{G}$, or simply a \textit{cover} when $\grp{G}$ and $f$ are clear from the context. In the next theorem, we give a sufficient condition for the existence of a cover with cardinality $k$. \\
\begin{theorem}\label{th:expcover}
Let $\# \grp{G}=q$, $f:\grp{G}\to \grp{G}$ and $V(f)=v$. Let $k$ be an integer such that $1\le k\le q$. If 
\begin{equation}\label{eq:expcover}
\begin{aligned}
q\sum_{i=0}^{k}(-1)^i\binom{k}{i}\frac{\binom{v}{i}}{\binom{q}{i}}<1,
\end{aligned}
\end{equation}then there exists a cover of $\grp{G}$ associate with $f$ with cardinality $k$. 
\end{theorem}
\begin{proof}
Let $\# \grp{G}=q$, $f:\grp{G}\to \grp{G}$ and $V(f)=v$. For $\alpha\in \grp{G}$, define $R_{\alpha}=\range(M_{\{\alpha\}})$, where range is defined as in (\ref{eq:def_rangeMS}). This is the set of indices of the rows of $M_f$ that has an entry $\alpha$. 

For $1\le k\le q$, let $\binom{\grp{G}}{k}$ denote the set of all $k$-subsets of $\grp{G}$. If $S\in\binom{\grp{G}}{k}$, then $\#\range(M_S)=\#\cup_{\alpha\in S}R_{\alpha}$, and we can rewrite it by the principle of inclusion and exclusion as 
\begin{equation}\label{eq:rows_union}
\begin{aligned}
\#\bigcup_{\alpha\in S}R_{\alpha}=\sum_{i=1}^{k}(-1)^{i-1}\sum_{T\subseteq S,\#T=i}\#\bigcap_{\alpha\in T}R_{\alpha}.
\end{aligned}
\end{equation}
Moreover, the number of rows of $M_{f}$ that contain $S$ can be written as 
\begin{equation}\label{eq:rows_intersection}
\#\bigcap_{\alpha\in S}R_{\alpha}=\sum_{r\in \grp{G}}\mathbbm{1}_{(S\subseteq \text{row }r)},
\end{equation}
where $\mathbbm{1}_{(S\subseteq \text{row }r)}$ denotes the indicator function that takes value $1$ if $S$ is a subset of row $r$ of $M_f$, and $0$ otherwise. 

For $1\le k\le q$, let $X_{k}=\#\bigcup_{\alpha\in S}R_{\alpha}$ and $Z_k=\#\bigcap_{\alpha\in S}R_{\alpha}$ be two random variables where $S$ is uniformly chosen from $\binom{\grp{G}}{k}$. By (\ref{eq:rows_intersection}), the expected value of $Z_{k}$ is
\begin{equation*}
\begin{aligned}
E(Z_k)&=\sum_{S\in\binom{\grp{G}}{k}}\Bigg(\sum_{r\in \grp{G}}\mathbbm{1}_{(S\subseteq \text{row }r)}\Bigg)\frac{1}{\binom{q}{k}}\\
&=\frac{1}{\binom{q}{k}}\sum_{r\in \grp{G}}\Bigg(\sum_{S\in\binom{\grp{G}}{k}}\mathbbm{1}_{(S\subseteq \text{row }r)}\Bigg)\\
&=\frac{1}{\binom{q}{k}}\sum_{r\in \grp{G}}\binom{v}{k}=\frac{\binom{v}{k}q}{\binom{q}{k}}, 
\end{aligned}
\end{equation*}
where the second last equality is because every row of $M_f$ has exactly $v$ entries. Hence, by (\ref{eq:rows_union}) the expected value of $X_k$ is 
\begin{equation}\label{eq:EXk}
\begin{aligned}
E(X_k)&=\sum_{i=1}^{k}(-1)^{i-1}\sum_{T\subseteq S,\#T=i}E(Z_i)\\
&=\sum_{i=1}^{k}(-1)^{i-1}\binom{k}{i}E(Z_i)\\
&=q\sum_{i=1}^{k}(-1)^{i-1}\binom{k}{i}\frac{\binom{v}{i}}{\binom{q}{i}}.
\end{aligned}
\end{equation}
Since $X_k\le q$ is an integer, if $E(X_k)>q-1$, then there exists $S\in \binom{\grp{G}}{k}$ such that $\#\bigcup_{\alpha\in S}R_{\alpha}=q$, which means $\range(M_S)=\grp{G}$ and $S$ is a cover. Since the summand in the last line of (\ref{eq:EXk}) equals $-1$ if $i=0$, we can simplify the condition $E(X_k)>q-1$ into (\ref{eq:expcover}) as 
\begin{equation*}
\begin{aligned}
1&>q-q\sum_{i=1}^{k}(-1)^{i-1}\binom{k}{i}\frac{\binom{v}{i}}{\binom{q}{i}}\\
&=q\sum_{i=0}^{k}(-1)^{i}\binom{k}{i}\frac{\binom{v}{i}}{\binom{q}{i}}. 
\end{aligned}
\end{equation*}
\end{proof}

\begin{remark}
When $V(f)=(q+1)/2$, for $1\le i\le k\le q$, 
\begin{equation*}
\begin{aligned}
E(Z_i)&=\frac{\binom{\frac{q+1}{2}}{i}q}{\binom{q}{i}}\\
&=\frac{q}{2^i}\frac{(q+1)(q-1)\cdots(q-2i+3)}{q(q-1)\cdots(q-i+1)}\\
&=\frac{q}{2^i}\prod_{j=0}^{i-1}\frac{q-(2j-1)}{q-j}.
\end{aligned}
\end{equation*}
When $q$ is sufficiently large relative to $k$, each factor $(q-(2j-1))/(q-j)$ in the above product is close to $1$. In this case, we may approximate each $E(Z_i)$ by $\frac{q}{2^i}$, and the inequality (\ref{eq:expcover}) becomes
\begin{equation*}
\begin{aligned}
1&>\sum_{i=0}^{k}(-1)^i\binom{k}{i}\frac{q}{2^i}=q\bigg(1-\frac{1}{2}\bigg)^k=\frac{q}{2^k},
\end{aligned}
\end{equation*}
which implies $k> \log_{2}q$. This suggests that when $q$ is large and $V(f)=(q+1)/2$, such as when $f(x)=x^2\in\ffx{q}$ for odd $q$, the search for a cover of $\ff{q}$ can start with small $k$-subsets where $k$ is about $\log_{2}q$. 
\end{remark}

When a cover $S$ is found, the following algorithm searches for an admissible subtable $\mathcal{A}$ of $M_f$ such that $\range(\mathcal{A})=\grp{G}$ and its value set is contained in $S$. 

\begin{algorithm}[H]\caption{Finding an admissible subtable whose value set is contained in a given cover $S$. }
\begin{itemize}
\item[\textbf{Input:}] $M_{f}=(m_{r,c}\,:\, r\in \grp{G},c\in\im{f})$, $S=\{v_{1},\dots,v_{k}\}$, some large integer $N$. 
\item[\textbf{Initialize:}] For every $c\in\im{f}$, let $A_{c}=\{\}$ and $R_{c}=\{r\in \grp{G}\,:\, m_{r,c}\in S\}$. For every $r\in \grp{G}$, let $a_{r}=\mathrm{FALSE}$ and $Q_{r}=\{c\in\im{f}\,:\, m_{r,c}\in S\}$. 
\item[\textbf{Step 1.}] For every $r\in \grp{G}$, find one $c\in Q_{r}$ such that $\#A_{c}\le \#\preim(f,c)$. Append $r$ to $A_{c}$ and update $a_{r}=\mathrm{TRUE}$. If no such $c$ exists, then no action at this step for $r$. 
\item[\textbf{Step 2.}] Let $B=\{r\in \grp{G}\,:\, a_{r}=\mathrm{FALSE}\}$. If $B\neq\emptyset$ then for every $r\in B$, append $r$ to some $A_{c}$ where $c\in Q_{r}$. Update $a_{r}=\mathrm{TRUE}$ and remove $r$ from $B$. If $\#A_{c}> \#\preim(f,c)$, then move some $r'\neq r$ from $A_{c}$ to $B$ and update $a_{r'}=\mathrm{FALSE}$. Repeat this step until $B=\emptyset$ or the number of iterations reaches $N$. 
\item[\textbf{Return:}] $B$, $\mathcal{A}=(m_{r,c}\,:\, c\in \im{f},r\in A_c)$. 
\end{itemize}
\end{algorithm}

The algorithm is based on the fact that the desired $\mathcal{A}$ must be a subtable of $M_S=(m_{r,c}\in S\,:\, m_{r,c}\in M_f)$. For each column $c$, let $R_c$ be the set of indices of rows $r$ such that $m_{r,c}\in S$. Similarly, for each row $r$, let $Q_r$ collects the indices of columns $c$ such that $m_{r,c}\in S$. So the set of all $R_c$ and $Q_r$ carry the information of $M_S$ where the search of admissible subtables should be limited in. The process starts by going over every row $r\in \grp{G}$ and appending one entry from $Q_r$ to $\mathcal{A}$. More precisely, for $c\in\im{f}$, we let $A_{c}$ be the set of indices of rows $r$ such that $m_{r,c}$ is chosen in $\mathcal{A}$. For every row $r$, find exactly one $c\in Q_r$ such that $\# A_{c}\le\#\preim(f,c)$. This means that $m_{r,c}\in M_S$ and column $c$ still has ``openings''. Assign $r$ to a $A_{c}$ and update $a_{r}=\mathrm{TRUE}$ to mark that $r$ has been assigned to some $A_{c}$. If no such $c$ exists, then we leave $r$ unassigned. 

After going over every row $r$ once in the first step, let $B$ be the set of indices of unassigned rows. If $B\neq \emptyset$, then we work through every row $r\in B$ and assign $r$ to some $A_{c}$ such that $c\in Q_{r}$, and remove $r$ from $B$. If this makes $\# A_{c}>\#\preim(f,c)$, then we move another element $r'$ from $A_{c}$ to $B$. Repeat this step until there are no more unassigned rows ($B=\emptyset$), or if the algorithm fails to assign every row within some large number $N$ of iterations. If the process succeeds with $B=\emptyset$, then $\bigcup_{c\in\im{f}}A_{c}=\grp{G}$, $\#A_{c}=\#\preim(f,c)$ for each $c\in\im{f}$, and the output $\mathcal{A}=(m_{r,c}\,:\, c\in\im{f},r\in A_{c})$ is an admissible subtable with $\range(\mathcal{A})=\grp{G}$ and value set contained in $S$. This also verifies that $\pres(f)\le \# S$. On the other hand, if $B\neq \emptyset$ when the algorithm terminates, then no conclusion can be drawn for $\pres(f)$ and one needs to restart with a different choice of $S$. 

\section{A Generalization of the Linear Integer Program for Optimizing DU}\label{sc:gen_IP}

In this final section, we generalize the IP in Section \ref{sc:alg_IP} to find a permutation of optimal DU in a given finite group.
As should be clear by now, one of the central motivations for studying
permutation resemblance is to better understand how to construct low DU bijections, so we feel the generalization to be a natural next question.
The generalized IP is the following: 


\begin{align}
&\text{minimize:} \qquad DU \notag\\
&\text{subject to:} \notag\\
& \sum_{c\in\grp{G}}x_{r,c}=1,  \text{ for all }r\in \grp{G},\label{eq:gen_IP_row}\\
& \sum_{r\in \grp{G}}x_{r,c}=1,  \text{ for all }c\in\grp{G},\label{eq:gen_IP_col}\\
&  \sum_{c\in \grp{G}}\sum_{r\in \grp{G}}z_{a,b,r,c}=\delta_{a,b},  \text{ for all }(a,b)\in\grp{G}^2,\,a\neq 0\label{eq:gen_IP_sum_delta}\\
&  2z_{a,b,r,c}\le x_{r+b,c+a}+x_{r,c}\le z_{a,b,r,c}+1, \text{ for all }(a,b,r,c)\in\grp{G}^4,\,a\neq 0\label{eq:gen_IP_z}\\
&  0\le\delta_{a,b}\le DU , \text{ for all }(a,b)\in\grp{G}^2,\,a\neq 0\label{eq:gen_IP_delta_DU}\\
&  x_{r,c},z_{a,b,r,c}\in\{0,1\}, \text{ for all }(a,b,r,c)\in\grp{G}^4,\,a\neq 0 \notag \\
&  DU,\delta_{a,b}\in\itr,\, \text{ for all }(a,b)\in \grp{G}^2,\,a\neq 0 \notag.
\end{align}

Let $f:\grp{G}\to\grp{G}$. We first need to redefine the subtraction table of $f$. For $c\in\im{f}$, recall that in the previous setting, the table $M_f$ only has the information of $\#\preim(f,c)$,  but does not keep track of the elements of $\preim(f,c)$. Also, each column $c\in\im{f}$ of an admissible subtable only records $g(\preim(f,c))$ and $(g+f)(\preim(f,c))$ as whole sets, but does not specify which element in $\preim(f,c)$ is to be mapped to which element. In order to compute DU directly in the IP, we need to keep more information on the preimages. We do this by ``decompressing'' each column. More precisely, for $1\le t\le u(f)$, the new table will now have $t$ copies of every column $c\in\im{f}$ in $P_t$, and each copy will be indexed by exactly one element of $\preim(f,c)$ instead. Then the columns of the decompressed subtraction table $M'_f$ are indexed by the whole group $\grp{G}$, and the indices of the columns are now to be viewed as the elements of the domain, instead of $\im{f}$. The entry of $M'_f$, using the new indices, is defined by $m_{r,c}:=r-f(c)$ for $r,c\in\grp{G}$. With this new setting, an admissible subtable $\mathcal{A}'$ should have exactly one entry in each column. Table \ref{tb:sub_x2_adm_decomp} is an example of $M'_f$ and $\mathcal{A}'$ obtained by decompressing Table \ref{tb:sub_x2_adm}. We keep the values $f(c)$ for reference, but the actual column indices are now the $c$'s. This new admissible subtable $\mathcal{A}'$ represents exactly one $g$ and $g+f$: for every grey entry $m_{r,c}\in\mathcal{A}'$, the corresponding $g$ sends $c$ to the entry value of $m_{r,c}$, and $g+f$ sends $c$ to $r$. For example, $g(0)=0$ and $(g+f)(0)=0$, $g(1)=\alpha$ and $(g+f)(1)=\alpha^2$, $g(-1)=1$ and $(g+f)(-1)=-1$, etc. The definition of $\range(\mathcal{A}')$ still makes sense and remains the same as before. To construct a permutation $g+f$, we still need to choose every row exactly once for an admissible subtable. 

\begin{table}[h]
\caption{\label{tb:sub_x2_adm_decomp} The decompressed subtraction table $M'_{f}$ of $f=x^2$ over $\ff{9}$ and an admissible subtable $\mathcal{A}'$. The first three columns and the first two rows label the coordinates of $M_{f}$. }
\begin{center}
\begin{tabular}{|c|c|*{9}{c}|}
\hline

 \multicolumn{2}{|c|}{$f(c)\rightarrow$}& \multicolumn{1}{|c|}{$0$} & \multicolumn{2}{|c|}{$1$} & \multicolumn{2}{|c|}{$\alpha^2$} & \multicolumn{2}{|c|}{$-1$} & \multicolumn{2}{|c|}{$-\alpha^2$} \\\hline
 \multicolumn{2}{|c|}{\multirow{2}{*}{\backslashbox{$r\searrow$}{$c\searrow$}}} & \multicolumn{1}{|c|}{$P_{1}$} &\multicolumn{8}{c|}{$P_{2}$}\\\cline{3-11}
\multicolumn{2}{|c|}{} & \multicolumn{1}{|c|}{$0$} & \multicolumn{1}{|c|}{$1$} & \multicolumn{1}{|c|}{$-1$} & \multicolumn{1}{|c|}{$\alpha$} & \multicolumn{1}{|c|}{$-\alpha$} & \multicolumn{1}{|c|}{$\alpha^2$} & \multicolumn{1}{|c|}{$-\alpha^2$} & \multicolumn{1}{|c|}{$\alpha^3$} & \multicolumn{1}{|c|}{$-\alpha^3$}\\\hline
$P_{1}$& $0$ & \cg$0$ & $-1$ & $-1$ & $-\alpha^2$ & $-\alpha^2$ & $1$ & $1$ & $\alpha^2$ & $\alpha^2$ \\\cline{1-2}
\multirow{4}{*}{$P_{2}$} & $1$ & $1$ & $0$ & $0$ & $-\alpha$ & $-\alpha$ & $-1$ & $-1$ & $-\alpha^3$ & $-\alpha^3$\\ 
 & $\alpha^2$ & $\alpha^2$ & \cg$\alpha$  & $\alpha$ & \cg$0$ & $0$ & $-\alpha^3$ & $-\alpha^3$ & $-\alpha^2$ & $-\alpha^2$ \\
 & $-1$ & $-1$ & $1$ & \cg$1$ & $\alpha^3$ & $\alpha^3$ & \cg$0$ & $0$ & $\alpha$ & $\alpha$\\
 & $-\alpha^2$ & $-\alpha^2$ & $\alpha^3$  & $\alpha^3$ & $\alpha^2$  & $\alpha^2$ & $-\alpha$ & $-\alpha$ & \cg$0$ & $0$\\\cline{1-2}
\multirow{4}{*}{$P_{0}$} & $\alpha$& $\alpha$ & $-\alpha^3$ & $-\alpha^3$ & $-1$ & $-1$ & $\alpha^2$ & $\alpha^2$ & $\alpha^3$ & $\alpha^3$\\ 
 & $\alpha^3$ & $\alpha^3$ & $-\alpha$ & $-\alpha$ & $\alpha$ & $\alpha$ & $-\alpha^2$ & $-\alpha^2$ & $-1$ & $-1$ \\
 & $-\alpha$ & $-\alpha$ & $-\alpha^2$ & $-\alpha^2$ & $-\alpha^3$ & $-\alpha^3$ & $\alpha^3$ & $\alpha^3$ & $1$ & \cg$1$\\
 & $-\alpha^3$ & $-\alpha^3$ & $\alpha^2$ & $\alpha^2$ & $1$ & \cg$1$ & $\alpha$ & \cg$\alpha$ & $-\alpha$ & $-\alpha$\\\hline
\end{tabular}
\end{center}
\end{table}

To construct a permutation $F:=g+f$ by IP, we once again associate every entry $m_{r,c}$ of $M'_f$ with a $\{0,1\}$-valued variable $x_{r,c}$. As before, $x_{r,c}=1$ if and only if $m_{r,c}\in\mathcal{A}'$, but in terms of functions, this now means $F(c)=r$. As mentioned, we require every row and every column to be chosen exactly once in $\mathcal{A}'$, in order to make sure that $F$ is injective and well-defined, respectively. These correspond to the constraints (\ref{eq:gen_IP_row}) and (\ref{eq:gen_IP_col}). 

Next, define an integer-value variable ``$DU$'' to record the DU of $F$. Clearly, the objective is to minimize $DU$. To compute $DU$, for every $(a,b)\in \grp{G}^{\star}\times\grp{G}$, define an integer-valued variable $\delta_{a,b}$ recording the number of solutions of $\Delta_{F,a}(x)=b$. Since by definition of DU, $DU$ is the maximum among all $\delta_{a,b}$, we add constraints (\ref{eq:gen_IP_delta_DU}). To compute each $\delta_{a,b}$, observe that $x$ satisfies $F(x+a)-F(x)=b$ if and only if there is another input $y$ such that 
\[y=x+a,\text{ and }F(y)=F(x)+b.\]
In other words, we have both $x_{F(x),x}=1$ and $x_{F(y),y}=x_{F(x)+b,x+a}=1$. Therefore, the definition of $\delta_{a,b}$ can be written into the following sum of products of binary variables. For a logical statement $S$, the indicator function $\mathbbm{1}_{(S)}=1$ if $S$ is true, and $\mathbbm{1}_{(S)}=0$ otherwise. 
\begin{equation}\label{eq:ext_IP_delta}
\begin{aligned}
\delta_{a,b}&=\#\{c\in\grp{G} \,:\, F(c+a)-F(c)=b\}\\
&=\sum_{c\in\grp{G}}\mathbbm{1}_{(F(c+a)-F(c)=b)}\\
&=\sum_{c\in\grp{G}}\mathbbm{1}_{(x_{F(c)+b,c+a}=1\text{ and }x_{F(c),c}=1)}\\
&=\sum_{c\in\grp{G}}\sum_{r\in\grp{G}}\mathbbm{1}_{(x_{r+b,c+a}=1\text{ and }x_{r,c}=1)}\\
&=\sum_{c\in\grp{G}}\sum_{r\in\grp{G}}x_{r+b,c+a}x_{r,c}. 
\end{aligned}
\end{equation}
To make the last expression of (\ref{eq:ext_IP_delta}) linear, we replace each $x_{r+b,c+a}x_{r,c}$ by a $\{0,1\}$-valued variable $z_{a,b,r,c}$, which yields (\ref{eq:gen_IP_sum_delta}), and add two inequalities (\ref{eq:gen_IP_z}) for each choice of $(a,b,r,c)\in\grp{G}^{4},a\neq 0$. For transforming polynomial constraints into linear ones, we refer the reader to \cite[Section 3.4]{chen10} and references therein. 

From the final formulation of this generalized IP, we can see that the optimal solutions do not rely on the choice of $f$. In fact, one can just take $f=0$ if the only requirements are the constructed function being a permutation and having lowest DU. Compare to the IP in Section \ref{sc:alg_IP}, this generalized IP is more costly to solve since it has both a larger number of variables and a larger number of constraints. That said, the generalized IP can do more combinations of different measurements of a function. For example, we can still compute $V(g)$ by adding $\{0,1\}$-valued variables $y_v$ for all $v\in\grp{G}$, an integer-value variable $V(g)=\sum_{v\in\grp{G}}y_v$, and constraints $x_{r,c}\le y_v$ for every $(r,c)\in\grp{G}^2$ and $r-f(c)=v$. Hence, we can restrict the problem further by setting the value of $V(g)$ to $\pres(f)$ if it is known, or require $V(g)$ in an interval of possible values for $\pres(f)$. Alternatively, instead of minimizing $DU$, we could also replace the objective by minimizing $V(g)$, and restricting $DU$ to being equal to, or bounded above by, a certain value. 

\bibliographystyle{amsplain}
\bibliography{mybibfile.bib}

\section*{Appendix: Some computational results}

As an application of the algorithm, we opted to examine the
$\pres$ of $x^d$ over $\ff{q}$ with $\gcd(d,q-1)>1$. We were particularly
interested in $\pres(x^2)$,
as $x^2$ is a planar function (has optimal DU) but is obviously not a PP.
We thus computed $\pres(x^2)$ for all finite fields of odd order up to $343$.
We also determined
$\pres(x^d)$, $\gcd(d,q-1)\neq 1$, for prime fields of order $\le 103$.
These results can be found in Tables \ref{tb:x2_Fp}, \ref{tb:x2_Fpe}, \ref{tb:xd_1mod4}, and \ref{tb:xd_3mod4} below.
For some readers, especially those working on the problem of finding
low DU bijetions, the results of this appendix may be the most
interesting part of the paper. 

The results in Table \ref{tb:x2_Fp} and \ref{tb:x2_Fpe} show that the value of $\pres(x^2)$ grows very slowly. Recall that in Section \ref{sc:alg_ex}, we narrowed down the candidates of the value set of an admissible subtable to the covers of $\ff{q}$.
In the remark after Theorem \ref{th:expcover}, we point out that when $q$ is large enough and $f=x^2$, it is possible to have a cover of $\ff{q}$ with size roughly $\log_{2}q$, so that for sufficiently large $q$ we have
$\pres(x^2)\le \log_{2}q$.
In fact, apart from the primes $p=5,13$ for $p\equiv 1\pmod 4$ and $p=3,7$ for $p\equiv 3\pmod 4$, every entry in Tables \ref{tb:x2_Fp} and \ref{tb:x2_Fpe} satisfies $\pres(x^2)\le \log_{2}q$. 

The functions $f=x^d$ with $\gcd(d,q-1)=d>1$ are $d$-to-$1$ over $\ffs{q}$, so $u(f)=d$ and $V(f)=(q-1)/d+1$. For this case, the bounds $u(f)\le \pres(f)\le q-V(f)+1$ established in \cite{prestheory} are 
\begin{equation*}\label{eq:xd_bounds}
d\le \pres(f)\le q - \frac{q-1}{d}. 
\end{equation*}
In Tables \ref{tb:xd_1mod4} and \ref{tb:xd_3mod4}, we determined $\pres(x^d)$ 
for all proper divisors of $p-1$, with $p\le 103$ a prime.
We also provide the values for the lower and upper bounds for comparison.
The tabulated data consistently shows that $u(f)$ is a much better estimate for $\pres(x^d)$ than the upper bound is, and for large $d$ (relative to the field size), $\pres(x^d)$ tends to be very close to the lower bound. This is, in some sense, what intuition
would suggest, but we do not have a proof of this in general.
That said, for $d=(p-1)/2$, the largest possible choice of $d<p-1$,
the bounds for $\pres(x^d)$ reduce to $d\le \pres(f)\le 2d-1$ and in
\cite{prestheory} it was shown that $\pres(x^d)$ is either $d$ or
$d+1$ for this case.

\begin{table}\caption{$\pres(x^2)$ over prime fields $\ff{p}$ for $p\le 337$}\label{tb:x2_Fp}
\begin{center}
\begin{tabular}[t]{|c|c|}
\hline
$p\equiv 1\pmod 4$	&	$\pres(x^2)$	\\\hline
5	&	3	\\\hline
13	&	4	\\\hline
17	&	4	\\\hline
29	&	4	\\\hline
37	&	4	\\\hline
41	&	5	\\\hline
53	&	5	\\\hline
61	&	5	\\\hline
73	&	5	\\\hline
89	&	5	\\\hline
97	&	5	\\\hline
101	&	5	\\\hline
109	&	5	\\\hline
113	&	5	\\\hline
137	&	5	\\\hline
149	&	5	\\\hline
157	&	5	\\\hline
173	&	5	\\\hline
181	&	5	\\\hline
193	&	5	\\\hline
197	&	5	\\\hline
229	&	5	\\\hline
\end{tabular}
\quad
\begin{tabular}[t]{|c|c|}
\hline
$p\equiv 1\pmod 4$	&	$\pres(x^2)$	\\\hline
233	&	5	\\\hline
241	&	5	\\\hline
257	&	6	\\\hline
269	&	5	\\\hline
277	&	6	\\\hline
281	&	6	\\\hline
293	&	5	\\\hline
313	&	6	\\\hline
317	&	5	\\\hline
337	&	6	\\\hline\hline
$p\equiv 3\pmod 4$	&	$\pres(x^2)$	\\\hline
3	&	2	\\\hline
7	&	3	\\\hline
11	&	3	\\\hline
19	&	4	\\\hline
23	&	4	\\\hline
31	&	4	\\\hline
43	&	4	\\\hline
47	&	4	\\\hline
59	&	5	\\\hline
67	&	5	\\\hline
71	&	5	\\\hline
\end{tabular}
\quad
\begin{tabular}[t]{|c|c|}
\hline
$p\equiv 3\pmod 4$	&	$\pres(x^2)$	\\\hline
79	&	5	\\\hline
83	&	5	\\\hline
103	&	4	\\\hline
107	&	5	\\\hline
127	&	5	\\\hline
131	&	5	\\\hline
139	&	5	\\\hline
151	&	5	\\\hline
163	&	5	\\\hline
167	&	5	\\\hline
179	&	5	\\\hline
191	&	5	\\\hline
199	&	5	\\\hline
211	&	5	\\\hline
223	&	5	\\\hline
227	&	5	\\\hline
239	&	5	\\\hline
251	&	5	\\\hline
263	&	5	\\\hline
271	&	5	\\\hline
307	&	5	\\\hline
311	&	5	\\\hline
331	&	6	\\\hline
\end{tabular}
\end{center}
\end{table}

\begin{table}\caption{$\pres(x^2)$ over $\ff{q}$ for prime powers $q\le 343$}\label{tb:x2_Fpe}
\begin{center}
\begin{tabular}{|c|c|}
\hline
$q$	&	$\pres(x^2)$	\\\hline
$9=3^2$	&	3	\\\hline
$27=3^3$	&	4	\\\hline
$81=3^4$	&	5	\\\hline
$243=3^5$	&	6	\\\hline
$25=5^2$	&	4	\\\hline
$125=5^3$	&	5	\\\hline
\end{tabular}
\quad
\begin{tabular}{|c|c|}
\hline
$q$	&	$\pres(x^2)$	\\\hline
$49=7^2$	&	5	\\\hline
$343=7^3$	&	6	\\\hline
$121=11^2$	&	5	\\\hline
$169=13^2$	&	5	\\\hline
$289=17^2$	&	6	\\\hline
\end{tabular}
\end{center}
\end{table}

\begin{table}\caption{$\pres(x^d)$ over $\ff{p}$ for $p\equiv 1\mod 4$ and non-equivalent $x^d$}\label{tb:xd_1mod4}
\begin{center}
\begin{tabular}[t]{*{5}{|c}|}
\hline
$p$	&	$d$	&	$u(x^d)$	&	$\pres(x^d)$	&	$p-(p-1)/d$	\\\hline	
5	&	2	&	2	&	3	&	3	\\\hline	\hline
13	&	2	&	2	&	4	&	7	\\\hline	
13	&	3	&	3	&	5	&	9	\\\hline	
13	&	4	&	4	&	5	&	10	\\\hline	
13	&	6	&	6	&	7	&	11	\\\hline	\hline
17	&	2	&	2	&	4	&	9	\\\hline	
17	&	4	&	4	&	5	&	13	\\\hline	
17	&	8	&	8	&	9	&	15	\\\hline	\hline
29	&	2	&	2	&	4	&	15	\\\hline	
29	&	4	&	4	&	6	&	22	\\\hline	
29	&	7	&	7	&	9	&	25	\\\hline	
29	&	14	&	14	&	15	&	27	\\\hline	\hline
37	&	2	&	2	&	4	&	19	\\\hline	
37	&	3	&	3	&	6	&	25	\\\hline	
37	&	4	&	4	&	7	&	28	\\\hline	
37	&	6	&	6	&	7	&	31	\\\hline	
37	&	9	&	9	&	9	&	33	\\\hline	
37	&	12	&	12	&	14	&	34	\\\hline	
37	&	18	&	18	&	19	&	35	\\\hline	\hline
41	&	2	&	2	&	5	&	21	\\\hline	
41	&	4	&	4	&	6	&	31	\\\hline	
41	&	5	&	5	&	7	&	33	\\\hline	
41	&	8	&	8	&	11	&	36	\\\hline	
41	&	10	&	10	&	12	&	37	\\\hline	
41	&	20	&	20	&	21	&	39	\\\hline	\hline
53	&	2	&	2	&	5	&	27	\\\hline	
53	&	4	&	4	&	7	&	40	\\\hline	
53	&	13	&	13	&	15	&	49	\\\hline	
53	&	26	&	26	&	27	&	51	\\\hline	\hline
61	&	2	&	2	&	5	&	31	\\\hline	
61	&	3	&	3	&	6	&	41	\\\hline	
61	&	4	&	4	&	6	&	46	\\\hline	
61	&	5	&	5	&	7	&	49	\\\hline	
61	&	6	&	6	&	9	&	51	\\\hline	
61	&	10	&	10	&	13	&	55	\\\hline	
61	&	12	&	12	&	15	&	56	\\\hline	
\end{tabular}
\quad
\begin{tabular}[t]{*{5}{|c}|}
\hline
$p$	&	$d$	&	$u(x^d)$	&	$\pres(x^d)$	&	$p-(p-1)/d$	\\\hline	
61	&	15	&	15	&	17	&	57	\\\hline	
61	&	20	&	20	&	23	&	58	\\\hline	
61	&	30	&	30	&	31	&	59	\\\hline	\hline
73	&	2	&	2	&	5	&	37	\\\hline	
73	&	3	&	3	&	6	&	49	\\\hline	
73	&	4	&	4	&	7	&	55	\\\hline	
73	&	6	&	6	&	9	&	61	\\\hline	
73	&	8	&	8	&	9	&	64	\\\hline	
73	&	9	&	9	&	11	&	65	\\\hline	
73	&	12	&	12	&	15	&	67	\\\hline	
73	&	18	&	18	&	19	&	69	\\\hline	
73	&	24	&	24	&	27	&	70	\\\hline	
73	&	36	&	36	&	37	&	71	\\\hline	\hline
89	&	2	&	2	&	5	&	45	\\\hline	
89	&	4	&	4	&	7	&	67	\\\hline	
89	&	8	&	8	&	12	&	78	\\\hline	
89	&	11	&	11	&	15	&	81	\\\hline	
89	&	22	&	22	&	24	&	85	\\\hline	
89	&	44	&	44	&	45	&	87	\\\hline	\hline
97	&	2	&	2	&	5	&	49	\\\hline	
97	&	3	&	3	&	7	&	65	\\\hline	
97	&	4	&	4	&	7	&	73	\\\hline	
97	&	6	&	6	&	10	&	81	\\\hline	
97	&	8	&	8	&	13	&	85	\\\hline	
97	&	12	&	12	&	16	&	89	\\\hline	
97	&	16	&	16	&	21	&	91	\\\hline	
97	&	24	&	24	&	25	&	93	\\\hline	
97	&	32	&	32	&	36	&	94	\\\hline	
97	&	48	&	48	&	49	&	95	\\\hline	\hline
101	&	2	&	2	&	5	&	51	\\\hline	
101	&	4	&	4	&	8	&	76	\\\hline	
101	&	5	&	5	&	8	&	81	\\\hline	
101	&	10	&	10	&	15	&	91	\\\hline	
101	&	20	&	20	&	26	&	96	\\\hline	
101	&	25	&	25	&	25	&	97	\\\hline	
101	&	50	&	50	&	51	&	99	\\\hline	
\end{tabular}
\end{center}
\end{table}

\begin{table}\caption{$\pres(x^d)$ over $\ff{p}$ for $p\equiv 3\mod 4$  and non-equivalent $x^d$}\label{tb:xd_3mod4}
\begin{center}
\begin{tabular}[t]{*{5}{|c}|}
\hline
$p$	&	$d$	&	$u(x^d)$	&	$\pres(x^d)$	&	$p-(p-1)/d$	\\\hline	
3	&	2	&	2	&	2	&	2	\\\hline	\hline
7	&	2	&	2	&	3	&	4	\\\hline	
7	&	3	&	3	&	3	&	5	\\\hline	\hline
11	&	2	&	2	&	3	&	6	\\\hline	
11	&	5	&	5	&	5	&	9	\\\hline	\hline
19	&	2	&	2	&	4	&	10	\\\hline	
19	&	3	&	3	&	5	&	13	\\\hline	
19	&	6	&	6	&	7	&	16	\\\hline	
19	&	9	&	9	&	9	&	17	\\\hline	\hline
23	&	2	&	2	&	4	&	12	\\\hline	
23	&	11	&	11	&	11	&	21	\\\hline	\hline
31	&	2	&	2	&	4	&	16	\\\hline	
31	&	3	&	3	&	5	&	21	\\\hline	
31	&	5	&	5	&	7	&	25	\\\hline	
31	&	6	&	6	&	8	&	26	\\\hline	
31	&	10	&	10	&	12	&	28	\\\hline	
31	&	15	&	15	&	15	&	29	\\\hline	\hline
43	&	2	&	2	&	4	&	22	\\\hline	
43	&	3	&	3	&	6	&	29	\\\hline	
43	&	6	&	6	&	9	&	36	\\\hline	
43	&	7	&	7	&	10	&	37	\\\hline	
43	&	14	&	14	&	17	&	40	\\\hline	
43	&	21	&	21	&	21	&	41	\\\hline	\hline
47	&	2	&	2	&	4	&	24	\\\hline	
47	&	23	&	23	&	23	&	45	\\\hline	\hline
59	&	2	&	2	&	5	&	30	\\\hline	
59	&	29	&	29	&	29	&	57	\\\hline	
\end{tabular}
\quad
\begin{tabular}[t]{*{5}{|c}|}
\hline
$p$	&	$d$	&	$u(x^d)$	&	$\pres(x^d)$	&	$p-(p-1)/d$	\\\hline	
67	&	2	&	2	&	5	&	34	\\\hline	
67	&	3	&	3	&	6	&	45	\\\hline	
67	&	6	&	6	&	10	&	56	\\\hline	
67	&	11	&	11	&	13	&	61	\\\hline	
67	&	22	&	22	&	26	&	64	\\\hline	
67	&	33	&	33	&	33	&	65	\\\hline	\hline
71	&	2	&	2	&	5	&	36	\\\hline	
71	&	5	&	5	&	8	&	57	\\\hline	
71	&	7	&	7	&	11	&	61	\\\hline	
71	&	10	&	10	&	14	&	64	\\\hline	
71	&	14	&	14	&	18	&	66	\\\hline	
71	&	35	&	35	&	35	&	69	\\\hline	\hline
79	&	2	&	2	&	5	&	40	\\\hline	
79	&	3	&	3	&	6	&	53	\\\hline	
79	&	6	&	6	&	11	&	66	\\\hline	
79	&	13	&	13	&	17	&	73	\\\hline	
79	&	26	&	26	&	29	&	76	\\\hline	
79	&	39	&	39	&	39	&	77	\\\hline	\hline
83	&	2	&	2	&	5	&	42	\\\hline	
83	&	41	&	41	&	41	&	81	\\\hline	\hline
103	&	2	&	2	&	4	&	52	\\\hline	
103	&	3	&	3	&	6	&	69	\\\hline	
103	&	6	&	6	&	11	&	86	\\\hline	
103	&	17	&	17	&	22	&	97	\\\hline	
103	&	34	&	34	&	37	&	100	\\\hline	
103	&	51	&	51	&	51	&	101	\\\hline	
\end{tabular}
\end{center}
\end{table}

\end{document}